\documentclass[12pt]{amsart}
\usepackage{amsmath,amsthm,amsfonts,amssymb,xcolor,enumitem,mathscinet}
\usepackage{graphicx}
\usepackage{tikz}
\usepackage{mathrsfs}
\usepackage{mathtools}
\usepackage[pdfpagelabels]{hyperref}
\usepackage[letterpaper,margin=1.1in]{geometry}

\newtheorem{theorem}{Theorem}[section]
\newtheorem{lemma}[theorem]{Lemma}
\newtheorem{corollary}[theorem]{Corollary}
\newtheorem{proposition}[theorem]{Proposition}

\theoremstyle{definition}
\newtheorem{definition}[theorem]{Definition}

\theoremstyle{remark}
\newtheorem{remark}[theorem]{Remark}
\newtheorem{example}[theorem]{Example}

\newcommand{\RR}{\mathbb{R}}
\newcommand{\QQ}{\mathbb{Q}}
\newcommand{\ZZ}{\mathbb{Z}}
\newcommand{\XX}{\mathbb{X}}
\newcommand{\cM}{\mathcal{M}}
\newcommand{\Leb}{\mathscr{L}}

\newcommand{\diam}{\mathop\mathrm{diam}\nolimits}
\newcommand{\side}{\mathop\mathrm{side}\nolimits}
\newcommand{\vol}{\mathop\mathrm{vol}\nolimits}
\newcommand{\res}{\hbox{ {\vrule height .22cm}{\leaders\hrule\hskip.2cm} }}
\newcommand{\gap}{\mathop\mathrm{gap}\nolimits}
\newcommand{\Child}{\mathsf{Child}}
\newcommand{\Collar}{\mathsf{Collar}}
\newcommand{\var}{\mathop\mathrm{var}\nolimits}

\numberwithin{equation}{section}

\begin{document}

\title{Method I revisited: extension, continuity, and applications}
\thanks{The author was partially supported by NSF DMS grant 2154047.}
\date{September 24, 2026}
\author{Matthew Badger}
\address{Department of Mathematics\\ University of Connecticut\\ Storrs, CT 06269-1009}
\email{matthew.badger@uconn.edu}

\begin{abstract}
Carath\'eodory's construction, also known as Method I, turns any weight on a family of sets into an outer measure. The difficulty is the extension problem: showing that the outer measure retains the prescribed weights. We record elementary properties of Method I, beginning with the fact that countable subadditivity on the covering family is exactly the criterion for faithful extension, and use them to prove standard results with only outer measures, $\sigma$-algebras, and measures, without premeasures on algebras. In topological spaces, a continuity principle reduces this criterion to finite subadditivity and finite approximation by open and compact sets. Applications include the Lebesgue integral as a measure, Tonelli's theorem, multidimensional Lebesgue--Stieltjes measures, Riesz representation for vector-valued functionals, Kolmogorov extension, mass distributions on nested partitions, and Frostman's lemma.
\end{abstract}

\maketitle

\setcounter{tocdepth}{1}
\tableofcontents

\section{Introduction}\label{s:intro}

Many constructions in measure theory begin with a nonnegative weight on a simple family of sets. Carath\'eodory's covering formula \cite{Caratheodory1914,Caratheodory1918}, which we call \emph{Method I} following Munroe \cite{Munroe} and Rogers \cite{Rogers}, turns such a weight into an outer measure automatically. The primary difficulty is the \emph{extension problem}: does the outer measure retain the prescribed values? For Lebesgue measure, the real work is proving that the outer measure assigns each rectangle its volume, a compactness argument that converts finite geometric information into countable subadditivity. Such verifications are sometimes compressed or delegated to exercises even though they contain the main construction.

Our approach is to pause after the covering formula and record the properties needed to solve the extension problem directly. Proposition \ref{p:method-one}(E) says that a weight $w$ on an arbitrary family $\mathcal E$ is preserved by the Method I outer measure if and only if $w$ is countably subadditive on $\mathcal E$. The rest of Proposition \ref{p:method-one} supplies measurability, approximation, and regularity, and maximality gives uniqueness under $\sigma$-finiteness. Remark \ref{r:caratheodory-hahn} recovers the Carath\'eodory--Hahn extension theorem and shows that the algebra and the premeasure are used only to verify these properties. Many standard constructions can then be carried out with outer measures, $\sigma$-algebras, and measures alone, without premeasures on algebras and without $\pi$-$\lambda$ or monotone class theorems.

Countable subadditivity must be checked against arbitrary countable covers. In topological spaces, the continuity principle, Lemma \ref{l:continuity}, reduces this to finite information: finite subadditivity, together with finite approximation from outside by open sets and from inside by compact sets, implies countable subadditivity.

We use this extension-first organization in several standard constructions. For an arbitrary nonnegative function $f$, the set function $A\mapsto\int_A f\,d\mu$ is shown directly to be a measure. Rectangle weights give product measures, Tonelli's theorem, and multidimensional Lebesgue--Stieltjes measures. A locally bounded linear functional on compactly supported vector-valued continuous functions gives a weight on open sets, from which the Riesz representation theorem follows. Cylinder weights give Kolmogorov's extension theorem. Compatible masses on nested partitions give measures with prescribed values, and localized net contents yield Frostman's lemma without weak limits of measures. These results are not offered as replacements for every use of premeasures or functional methods, but they show that a substantial collection of basic constructions can be proved from the Method I extension principle.

This project began during the author's sabbatical in Fall 2020 while preparing notes for a graduate measure theory course at the University of Connecticut. The initial aim was to organize Carath\'eodory's covering construction as a direct method for solving extension problems. Later, a mass-distribution argument written with Schul \cite{BS4} used a compactness step to pass from finite to countable information. Comparing that argument with the standard construction of Lebesgue measure led to the continuity principle used below. The author's perspective on measure theory was shaped in particular by Folland \cite{Folland}, Royden \cite{Royden}, Evans and Gariepy \cite{EvansGariepy}, Mattila \cite{Mattila}, and Rogers \cite{Rogers}.

\section{Method I outer measures}\label{s:method-one}

Let $\XX$ be a nonempty set. A \emph{measure} on a $\sigma$-algebra $\cM$ on $\XX$ is a countably additive function $\mu:\cM\to[0,\infty]$ with $\mu(\emptyset)=0$. An \emph{outer measure} on $\XX$ is a function $\mu:\mathcal P(\XX)\to[0,\infty]$ with $\mu(\emptyset)=0$ that is monotone and countably subadditive, i.e.~$\mu(A)\leq\sum_{i=1}^\infty\mu(A_i)$ whenever $A\subset\bigcup_{i=1}^\infty A_i$. For every set $A\subset\XX$, the \emph{restriction} $\mu\res A$ of $\mu$ to $A$ is the outer measure defined by $\mu\res A(S)=\mu(S\cap A)$ for all $S\subset\XX$. A set $A\subset\XX$ is \emph{$\mu$ measurable} if
\[
 \mu=\mu\res A+\mu\res(\XX\setminus A),
\]
i.e.~if $\mu(S)=\mu(S\cap A)+\mu(S\setminus A)$ for all $S\subset\XX$.

For a nonempty family $\mathcal E\subset\mathcal P(\XX)$, let $\mathcal E_\sigma$ and $\mathcal E_\delta$ denote the families of countable unions and countable intersections of sets in $\mathcal E$, respectively, and put $\mathcal E_{\sigma\delta}:=(\mathcal E_\sigma)_\delta$.

\begin{lemma}[Method I]\label{l:method-one}
Let $\XX$ be a nonempty set and let $\mathcal E\subset\mathcal P(\XX)$ be a family of sets with $\emptyset\in\mathcal E$. For every function $w:\mathcal E\to[0,\infty]$ with $w(\emptyset)=0$, the function $\mu_w:\mathcal P(\XX)\to[0,\infty]$ defined by
\begin{equation}\label{method-one}
 \mu_w(A):=\inf\left\{\sum_{i=1}^\infty w(E_i):E_1,E_2,\dots\in\mathcal E,\ A\subset\bigcup_{i=1}^\infty E_i\right\}\quad\text{for all }A\subset\XX
\end{equation}
is an outer measure on $\XX$, where $\inf\emptyset=\infty$.
\end{lemma}

\begin{proof}
The cover $\emptyset,\emptyset,\dots$ gives $\mu_w(\emptyset)=0$. If $A\subset B$, then every cover of $B$ is a cover of $A$, and $\mu_w(A)\leq\mu_w(B)$. Let $A\subset\bigcup_{j=1}^\infty A_j$, and assume that $\sum_j\mu_w(A_j)<\infty$. Given $\varepsilon>0$, choose covers $E^j_1,E^j_2,\dots\in\mathcal E$ of $A_j$ with $\sum_iw(E^j_i)\leq\mu_w(A_j)+2^{-j}\varepsilon$. Together, they cover $A$, and $\mu_w(A)\leq\sum_j\mu_w(A_j)+\varepsilon$.
\end{proof}

We call $w$ a \emph{weight} and $\mu_w$ the \emph{Method I outer measure} with weight $w$.

\begin{proposition}[properties of Method I outer measures]\label{p:method-one}
Let $\XX$, $\mathcal E$, $w$, and $\mu_w$ be as in Lemma \ref{l:method-one}.
\begin{enumerate}[leftmargin=3.2em,labelwidth=2.5em,labelsep=.7em]
\item[\textup{(E)}] (extension) If $w$ is countably subadditive on $\mathcal E$, i.e.
\[
 E,E_1,E_2,\dots\in\mathcal E\quad\text{and}\quad E\subset\bigcup_{i=1}^\infty E_i\quad\text{imply}\quad w(E)\leq\sum_{i=1}^\infty w(E_i),
\]
then $\mu_w(E)=w(E)$ for all $E\in\mathcal E$.
\item[\textup{(M)}] (measurability) Suppose that for all $E,E'\in\mathcal E$, there are $F_0,F_1,\dots,F_p\in\mathcal E$ such that
\[
 E'\cap E=F_0,\qquad E'\setminus E=F_1\cup\dots\cup F_p,\quad\text{and}\quad\sum_{j=0}^pw(F_j)\leq w(E').
\]
Then every set in $\mathcal E$ is $\mu_w$ measurable.
\item[\textup{(OA)}] (outer approximation) For every $S\subset\XX$,
\[
 \mu_w(S)=\inf\{\mu_w(B):S\subset B,\ B\in\mathcal E_\sigma\}.
\]
\item[\textup{(IA)}] (inner approximation) For every $\mu_w$ measurable set $A\subset\XX$ with $\mu_w(A)<\infty$ and every $\varepsilon>0$, there are $B,C\in\mathcal E_\sigma$ such that
\[
 B\setminus C\subset A\qquad\text{and}\qquad
 \mu_w\bigl(A\setminus(B\setminus C)\bigr)<\varepsilon.
\]
\item[\textup{(OR)}] (outer regularity) If $S\subset\XX$ and $\mu_w(S)<\infty$, there is $B\in\mathcal E_{\sigma\delta}$ with $S\subset B$ and $\mu_w(B)=\mu_w(S)$. If $\XX\in\mathcal E_\sigma$, the same conclusion holds for every $S\subset\XX$.
\item[\textup{(IR)}] (inner regularity) For every $\mu_w$ measurable set $A\subset\XX$ with $\mu_w(A)<\infty$, there are $C_1,C_2\in\mathcal E_{\sigma\delta}$ such that $C_1\setminus C_2\subset A$, $\mu_w(C_1\setminus C_2)=\mu_w(C_1)=\mu_w(A)$, and $\mu_w(C_2)=0$.
\end{enumerate}
\end{proposition}

\begin{proof}
Suppose that $w$ is countably subadditive on $\mathcal E$, and let $E\in\mathcal E$. Every countable cover of $E$ by sets in $\mathcal E$ has total weight at least $w(E)$. Hence $w(E)\leq\mu_w(E)$. The cover $E,\emptyset,\emptyset,\dots$ gives $\mu_w(E)\leq w(E)$, which gives (E).

Assume the hypothesis of (M), and let $E\in\mathcal E$ and $S\subset\XX$. Let $E_1,E_2,\dots\in\mathcal E$ cover $S$, and for each $i$ let $F^i_0,\dots,F^i_{p_i}$ be given by the hypothesis with $E'=E_i$. Since $\mu_w(F)\leq w(F)$ for all $F\in\mathcal E$,
\[
 \mu_w(S\cap E)+\mu_w(S\setminus E)\leq\sum_{i=1}^\infty\bigl(\mu_w(E_i\cap E)+\mu_w(E_i\setminus E)\bigr)\leq\sum_{i=1}^\infty\sum_{j=0}^{p_i}w(F^i_j)\leq\sum_{i=1}^\infty w(E_i).
\]
Taking the infimum over all covers gives $\mu_w(S\cap E)+\mu_w(S\setminus E)\leq\mu_w(S)$. The reverse inequality holds by subadditivity, and $E$ is $\mu_w$ measurable.

For (OA), if $S\subset B\in\mathcal E_\sigma$, then $\mu_w(S)\leq\mu_w(B)$. Conversely, every cover $S\subset\bigcup_iE_i$ with $E_i\in\mathcal E$ gives $B:=\bigcup_iE_i\in\mathcal E_\sigma$ and
\[
 \mu_w(B)\leq\sum_iw(E_i).
\]
Taking infima proves (OA).

For (IA), let $A$ be $\mu_w$ measurable with $\mu_w(A)<\infty$, and let $\varepsilon>0$. By (OA), choose $B\in\mathcal E_\sigma$ with $A\subset B$ and $\mu_w(B)<\mu_w(A)+\varepsilon/2$. Since $A$ is $\mu_w$ measurable, $\mu_w(B\setminus A)<\varepsilon/2$. Applying (OA) again, choose $C\in\mathcal E_\sigma$ with $B\setminus A\subset C$ and $\mu_w(C)<\varepsilon$. Then $B\setminus C\subset A$ and $A\setminus(B\setminus C)=A\cap C\subset C$, which gives (IA).

If $\mu_w(S)<\infty$, then by (OA), for every $n\geq1$ choose $A_n\in\mathcal E_\sigma$ with $S\subset A_n$ and $\mu_w(A_n)\leq\mu_w(S)+1/n$. Then
\[
 B:=\bigcap_{n=1}^\infty A_n\in\mathcal E_{\sigma\delta},\qquad S\subset B,
\]
and $\mu_w(S)\leq\mu_w(B)\leq\mu_w(A_n)\leq\mu_w(S)+1/n$ for every $n$. If $\mu_w(S)=\infty$ and $\XX\in\mathcal E_\sigma$, take $B=\XX$. This proves (OR).

Let $A$ be $\mu_w$ measurable with $\mu_w(A)<\infty$. By the finite-measure case of (OR), choose $C_1\in\mathcal E_{\sigma\delta}$ with $A\subset C_1$ and $\mu_w(C_1)=\mu_w(A)$. Since $A$ is $\mu_w$ measurable,
\[
 \mu_w(C_1)=\mu_w(A)+\mu_w(C_1\setminus A),
\]
so $\mu_w(C_1\setminus A)=0$. Applying (OR) once more, choose $C_2\in\mathcal E_{\sigma\delta}$ with $C_1\setminus A\subset C_2$ and $\mu_w(C_2)=0$. Then $C_1\setminus C_2\subset A$ and
\[
 \mu_w(C_1)\leq\mu_w(C_1\setminus C_2)+\mu_w(C_2)=\mu_w(C_1\setminus C_2)\leq\mu_w(C_1),
\]
which gives (IR).
\end{proof}

\begin{remark}[maximal extension]\label{r:extension}
Proposition \ref{p:method-one}(E) is a characterization. If an outer measure $\nu$ satisfies $\nu(E)=w(E)$ for every $E\in\mathcal E$, then $w$ is countably subadditive on $\mathcal E$. Moreover, every cover $S\subset\bigcup_iE_i$ gives
\[
 \nu(S)\leq\sum_i\nu(E_i)=\sum_iw(E_i),
\]
so $\nu\leq\mu_w$. Thus, $w$ extends to an outer measure if and only if it is countably subadditive, and in that case $\mu_w$ is the largest outer measure extending $w$. Countable subadditivity already includes monotonicity by using the cover $E',\emptyset,\emptyset,\dots$ when $E\subset E'$.
\end{remark}

\begin{remark}[measurability criterion]\label{r:measurability} The following characterization of $\mu_w$ measurability holds for every weight $w$. A set $A\subset\XX$ is $\mu_w$ measurable if and only if
\begin{equation}\label{measurability-test}
 \mu_w(E\cap A)+\mu_w(E\setminus A)\leq w(E)\quad\text{for all }E\in\mathcal E.
\end{equation}
If $A$ is $\mu_w$ measurable, then the left side of \eqref{measurability-test} equals $\mu_w(E)\leq w(E)$. Conversely, if \eqref{measurability-test} holds and $E_1,E_2,\dots\in\mathcal E$ cover $S\subset\XX$, then
\[
 \mu_w(S\cap A)+\mu_w(S\setminus A)\leq\sum_{i=1}^\infty\bigl(\mu_w(E_i\cap A)+\mu_w(E_i\setminus A)\bigr)\leq\sum_{i=1}^\infty w(E_i),
\]
and taking the infimum over covers gives $\mu_w(S\cap A)+\mu_w(S\setminus A)\leq\mu_w(S)$. The hypothesis of Proposition \ref{p:method-one}(M) is one convenient way to verify \eqref{measurability-test} when $A\in\mathcal E$. It is modeled on the semiring axioms, which require $E'\cap E\in\mathcal E$ and $E'\setminus E$ to be a finite disjoint union of sets in $\mathcal E$. Proposition \ref{p:method-one}(M) asks for less. The sets $F_1,\dots,F_p$ need not be disjoint, only an inequality between weights is required, and \eqref{measurability-test} allows arbitrary countable covers. Thus, measurability of the sets in $\mathcal E$ is again a property of the outer measure $\mu_w$, checked on $\mathcal E$, rather than a structural requirement on the family $\mathcal E$.
\end{remark}

An outer measure $\nu$ on $\XX$ is \emph{regular} if for every $S\subset\XX$ there is a $\nu$ measurable set $A\supset S$ such that $\nu(A)=\nu(S)$. A measure $\mu$ on a measurable space $(\XX,\cM)$ is \emph{complete} if every subset of a $\mu$ null set in $\cM$ belongs to $\cM$. The \emph{completion} $(\XX,\overline{\cM},\overline\mu)$ of $(\XX,\cM,\mu)$ is obtained by adjoining all subsets of measurable $\mu$ null sets: $\overline{\cM}$ consists of the sets $A\cup N'$ with $A\in\cM$ and $N'\subset N$ for some $N\in\cM$ with $\mu(N)=0$, and $\overline\mu(A\cup N')=\mu(A)$. We say that $S\subset\XX$ is \emph{$\mu$-locally in $\cM$} if
\[
 A\in\cM\quad\text{and}\quad \mu(A)<\infty
 \qquad\Longrightarrow\qquad
 S\cap A\in\cM,
\]
and that $\mu$ is \emph{saturated} if every set that is $\mu$-locally in $\cM$ belongs to $\cM$.\footnote{The standard phrase is ``locally $\mu$ measurable''. See, for example, Folland \cite[Exercise 1.16]{Folland} and Royden \cite{Royden}.} The sets that are $\mu$-locally in $\cM$ form a $\sigma$-algebra containing $\cM$. The \emph{saturation} of $\mu$ is the extension of $\mu$ to this $\sigma$-algebra that assigns the value $\infty$ to every set not in $\cM$. It is a measure: if pairwise disjoint sets $A_1,A_2,\dots$ that are $\mu$-locally in $\cM$ have union $A\in\cM$ with $\mu(A)<\infty$, then every $A_i=A_i\cap A$ belongs to $\cM$. Every $\sigma$-finite measure is saturated. We will need the saturation of the completion below.

If $\nu$ is an outer measure, let $\cM_\nu$ denote the class of $\nu$ measurable sets. If $\mu$ is a measure on a $\sigma$-algebra $\cM$, let $\mu^*$ be the Method I outer measure with $\mathcal E=\cM$ and $w=\mu$. Put
\[
 \nu^-:=\nu|_{\cM_\nu}\qquad\text{and}\qquad\mu^+:=(\mu^*)^-.
\]
Thus, the superscript $-$ restricts an outer measure to its measurable sets, while the superscript $+$ enlarges the $\sigma$-algebra of a measure by completion and saturation.

\begin{theorem}[the big correspondence]\label{t:measure-outer-correspondence}
Let $\XX$ be a nonempty set.
\begin{enumerate}
\item (Carath\'eodory) For every outer measure $\nu$ on $\XX$, the class $\cM_\nu$ is a $\sigma$-algebra and $\nu^-$ is a complete measure.

\item If $\nu$ is regular, then $\nu^-$ is saturated and
\[
 (\nu^-)^*=\nu.
\]

\item If $\mu$ is a measure on a measurable space $(\XX,\cM)$, then $\mu^*$ is a regular outer measure, every set in $\cM$ is $\mu^*$ measurable, $\mu^*|_\cM=\mu$, and $\mu^*$ is the largest outer measure extending $\mu$. Moreover, $\mu^+$ is the saturation of the completion of $\mu$.

\item Consequently, the maps $\mu\mapsto\mu^*$ and $\nu\mapsto\nu^-$ are inverse to one another on the classes
\[
 \left\{\begin{array}{c}
 \text{complete, saturated measures}\\[-1mm]
 \text{on measurable spaces over }\XX
 \end{array}\right\}
 \longleftrightarrow
 \left\{\begin{array}{c}
 \text{regular outer measures}\\[-1mm]
 \text{on }\XX
 \end{array}\right\}.
\]

\item Restricting to $\sigma$-finite objects gives
\[
 \left\{\begin{array}{c}
 \sigma\text{-finite complete measures}\\[-1mm]
 \text{on measurable spaces over }\XX
 \end{array}\right\}
 \longleftrightarrow
 \left\{\begin{array}{c}
 \sigma\text{-finite regular outer measures}\\[-1mm]
 \text{on }\XX
 \end{array}\right\},
\]
where a regular outer measure is called $\sigma$-finite when its induced measure is $\sigma$-finite. In particular, if $\mu$ is $\sigma$-finite, then $\mu^+$ is the completion of $\mu$.
\end{enumerate}
\end{theorem}

\begin{proof}
Part (1) is Carath\'eodory's theorem. See, for example, Rogers \cite[Chapter 1]{Rogers} or Folland \cite[Chapter 1]{Folland}.

Suppose that $\nu$ is regular and that $S$ is $\nu^-$-locally in $\cM_\nu$. To verify the Carath\'eodory inequality, it is enough to consider $F\subset\XX$ with $\nu(F)<\infty$. Choose $A\in\cM_\nu$ with $F\subset A$ and $\nu(A)=\nu(F)$. Since $S\cap A\in\cM_\nu$,
\[
 \nu(F)=\nu(A)=\nu(A\cap S)+\nu(A\setminus S)
 \geq \nu(F\cap S)+\nu(F\setminus S).
\]
The reverse inequality is subadditivity. Thus, $S\in\cM_\nu$, and $\nu^-$ is saturated. Since $\nu$ extends $\nu^-$, maximality in Remark \ref{r:extension} gives $\nu\leq(\nu^-)^*$. Conversely, for $S\subset\XX$, regularity gives $A\in\cM_\nu$ with $S\subset A$ and $\nu(A)=\nu(S)$, whence
$ (\nu^-)^*(S)\leq(\nu^-)^*(A)=\nu^-(A)=\nu(S). $
This proves (2).

For (3), the measure $\mu$ is countably subadditive on $\cM$, and the sets $F_0=E'\cap E$ and $F_1=E'\setminus E$ satisfy the hypothesis of Proposition \ref{p:method-one}(M). Hence Proposition \ref{p:method-one}(E) and (M) give extension and measurability, Remark \ref{r:extension} gives maximality, and Proposition \ref{p:method-one}(OR), applied with $\mathcal E=\cM$, gives regularity because $\XX\in\cM$. Let $(\XX,\overline{\cM},\overline\mu)$ be the completion of $(\XX,\cM,\mu)$. Every set in $\overline{\cM}$ is $\mu^*$ measurable: sets in $\cM$ are measurable by Proposition \ref{p:method-one}(M), and every subset of a $\mu^*$ null set is measurable by part (1).

We claim that $\cM_{\mu^*}$ is exactly the saturation of $\overline{\cM}$. First let $A\in\cM_{\mu^*}$ and let $B\in\cM$ with $\mu(B)<\infty$. The set $A\cap B$ is $\mu^*$ measurable and has finite outer measure. Proposition \ref{p:method-one}(IR), again with $\mathcal E=\cM$, gives $C_1,C_2\in\cM$ such that
\[
 D:=C_1\setminus C_2\subset A\cap B,\qquad
 \mu^*(D)=\mu^*(A\cap B),\qquad \mu(C_2)=0.
\]
Hence $(A\cap B)\setminus D$ is $\mu^*$ null. By Proposition \ref{p:method-one}(OR), it is contained in a set $N\in\cM$ with $\mu(N)=0$. Thus, $A\cap B\in\overline{\cM}$. The same follows for every finite-measure $B\in\overline{\cM}$ after changing $B$ by a measurable null set. Therefore, $A$ is $\overline\mu$-locally in $\overline{\cM}$.

Conversely, suppose that $A$ is $\overline\mu$-locally in $\overline{\cM}$. Let $S\subset\XX$ with $\mu^*(S)<\infty$. By regularity, choose $B\in\cM$ with $S\subset B$ and $\mu(B)=\mu^*(S)$. Then $A\cap B\in\overline{\cM}$, hence is $\mu^*$ measurable, and
\[
 \mu^*(S) = \mu^*(B)
 =\mu^*(B\cap A)+\mu^*(B\setminus A)
 \geq \mu^*(S\cap A)+\mu^*(S\setminus A).
\]
Subadditivity gives the reverse inequality, and the case $\mu^*(S)=\infty$ is automatic. Thus, $A\in\cM_{\mu^*}$. Finally, a set in this saturation that is not in $\overline{\cM}$ must have $\mu^*$ measure $\infty$, since otherwise the preceding regularity argument would place it in $\overline{\cM}$. Hence $\mu^+$ is precisely the saturation of the completion.

If $\mu$ is complete and saturated, then $\mu^+=\mu$, so (3) gives $(\mu^*)^-=\mu$. Together with (2), this proves (4). If $\mu$ is $\sigma$-finite, its completion is $\sigma$-finite and saturated, so (3) gives the last assertion of (5). Conversely, the induced measure of a $\sigma$-finite regular outer measure is $\sigma$-finite by definition. This proves (5).
\end{proof}

Munroe \cite[Theorems 12.3 and 12.4]{Munroe} records parts of Proposition \ref{p:method-one} and Theorem \ref{t:measure-outer-correspondence}. Theorem 12.3 contains the $\mathcal E_{\sigma\delta}$ outer regularity in (OR), and Theorem 12.4 treats extension, regularity, completion, and reconstruction when the weight is already a measure. Proposition \ref{p:method-one}(E) and Remark \ref{r:extension} add that an arbitrary weight extends faithfully exactly when it is countably subadditive on the covering family. The regularity hypothesis in Theorem \ref{t:measure-outer-correspondence}(2) cannot be dropped. For every outer measure $\nu$, maximality in Remark \ref{r:extension} gives only $(\nu^-)^*\geq\nu$, and Kim \cite{Kim1985} gives examples in which the inequality is strict.

Given a family $\mathcal E$ of subsets of $\XX$, write $\cM(\mathcal E)$ for the $\sigma$-algebra generated by $\mathcal E$.
Call a weight $w$ on $\mathcal E$ \emph{$\sigma$-finite} if $\XX$ is covered by sets $E_1,E_2,\dots\in\mathcal E$ with $w(E_i)<\infty$ for every $i$.

\begin{lemma}[uniqueness]\label{l:uniqueness}
Let $\mathcal E$, $w$, and $\mu_w$ be as in Lemma \ref{l:method-one}, and suppose that $w$ satisfies the hypotheses of Proposition \ref{p:method-one}(E) and (M). If $w$ is $\sigma$-finite, then the restriction of $\mu_w$ to $\cM(\mathcal E)$ is the unique measure on $\cM(\mathcal E)$ that agrees with $w$ on $\mathcal E$.
\end{lemma}

\begin{proof}
Choose $E_1,E_2,\dots\in\mathcal E$ with $\XX=\bigcup_iE_i$ and $w(E_i)<\infty$ for all $i$. By Proposition \ref{p:method-one}(E) and (M) and Theorem \ref{t:measure-outer-correspondence}(1), $\mu_w$ is a measure on $\cM(\mathcal E)$ that extends $w$. Let $\nu$ be another such measure. As in Remark \ref{r:extension}, countable subadditivity of $\nu$ gives $\nu\leq\mu_w$ on $\cM(\mathcal E)$. Fix $i$. Since $\nu(E_i)=w(E_i)=\mu_w(E_i)<\infty$,
\[
 \nu(A\cap E_i)=w(E_i)-\nu(E_i\setminus A)\geq w(E_i)-\mu_w(E_i\setminus A)=\mu_w(A\cap E_i)\quad\text{for all }A\in\cM(\mathcal E).
\]
Hence $\nu$ and $\mu_w$ agree on every set in $\cM(\mathcal E)$ contained in some $E_i$. Writing $A$ as the disjoint union of the sets $A\cap E_i\setminus(E_1\cup\dots\cup E_{i-1})$ completes the proof.
\end{proof}

\begin{lemma}[measure criterion]\label{l:measure-criterion}
Let $\cM$ be a $\sigma$-algebra on a nonempty set $\XX$, and let $\mu:\cM\to[0,\infty]$ satisfy $\mu(\emptyset)=0$. If $\mu$ is countably subadditive on $\cM$ and $\mu(A\cup B)=\mu(A)+\mu(B)$ for all disjoint $A,B\in\cM$, then $\mu$ is a measure on $\cM$.
\end{lemma}

\begin{proof}
Let $\mu_w$ be the Method I outer measure with $\mathcal E=\cM$ and $w=\mu$. By Proposition \ref{p:method-one}(E), $\mu_w|_\cM=\mu$. For $E,E'\in\cM$, take $F_0=E'\cap E$ and $F_1=E'\setminus E$. Then $w(F_0)+w(F_1)=w(E')$, and every set in $\cM$ is $\mu_w$ measurable by Proposition \ref{p:method-one}(M). By Theorem \ref{t:measure-outer-correspondence}(1), $\mu=\mu_w|_\cM$ is countably additive. (Alternatively, finite additivity gives monotonicity, and for disjoint $A_1,A_2,\dots\in\cM$ gives $\mu(\bigcup_{i=1}^\infty A_i)\geq\sum_{i=1}^n\mu(A_i)$ for all $n$. The reverse inequality is countable subadditivity.)
\end{proof}

\begin{remark}[the Carath\'eodory--Hahn extension theorem]\label{r:caratheodory-hahn}
Let $\mathcal A$ be an algebra of subsets of $\XX$ (a family containing $\emptyset$ that is closed under complements and finite unions), and let $\mu_0:\mathcal A\to[0,\infty]$ be a premeasure: $\mu_0(\emptyset)=0$ and $\mu_0(\bigcup_iA_i)=\sum_i\mu_0(A_i)$ whenever $A_1,A_2,\dots\in\mathcal A$ are pairwise disjoint with $\bigcup_iA_i\in\mathcal A$. The Carath\'eodory--Hahn theorem says that $\mu_0$ extends to a measure on $\cM(\mathcal A)$, uniquely when $\XX$ is a countable union of sets in $\mathcal A$ of finite $\mu_0$ measure. Take $\mathcal E=\mathcal A$ and $w=\mu_0$. A cover $A\subset\bigcup_iA_i$ in $\mathcal A$ can be disjointified inside $\mathcal A$, which gives countable subadditivity, and the sets $F_0=E'\cap E$ and $F_1=E'\setminus E$ verify the hypothesis of Proposition \ref{p:method-one}(M). Proposition \ref{p:method-one}(E), Theorem \ref{t:measure-outer-correspondence}(1), and Lemma \ref{l:uniqueness} then give existence and uniqueness. The same argument applies to semirings, after writing the differences as finite disjoint unions.

Thus, the algebra and the countable additivity of $\mu_0$ are used only to verify the hypotheses of Proposition \ref{p:method-one}(E) and (M). The applications below verify these hypotheses directly on the family where a weight is naturally defined, usually a semiring, and the weight is never extended to the generated ring. In Theorem \ref{t:riesz}, the family is not a semiring, and measurability comes from Remark \ref{r:measurability}.
\end{remark}

\begin{example}[a weight that is not countably subadditive]\label{ex:rational-intervals}
Let $\XX=\QQ\cap(0,1]$, let $\mathcal E$ consist of the empty set and the sets $\QQ\cap(a,b]$ with $0\leq a<b\leq1$, and put $w(\QQ\cap(a,b])=b-a$. Since the rational numbers are dense, $w$ is well defined, and if $\QQ\cap(a,b]\subset\bigcup_{i=1}^k\QQ\cap(a_i,b_i]$, then density gives $(a,b]\subset\bigcup_{i=1}^k[a_i,b_i]$. If finitely many closed intervals $[a_i,b_i]$ cover $(a,b]$, then $b-a\leq\sum_{i=1}^k(b_i-a_i)$. Thus, $w$ is finitely subadditive. Enumerate $\XX=\{q_1,q_2,\dots\}$ and let $\varepsilon>0$. The sets $\QQ\cap(\max\{q_j-2^{-j}\varepsilon,0\},q_j]$ cover $\XX$ and have total weight at most $\varepsilon$. Hence $\mu_w(\XX)=0<1=w(\XX)$, and $w$ is not countably subadditive.
\end{example}

\section{Integration}\label{s:integration}

\subsection{The Lebesgue integral as a measure}\label{ss:integral}

For $A\subset\XX$, let $\chi_A$ denote the characteristic function of $A$, so $\chi_A(x)=1$ for $x\in A$ and $\chi_A(x)=0$ for $x\notin A$.

\begin{definition}[the Lebesgue integral as a set function]\label{d:integral}
Let $\mu$ be a measure on a $\sigma$-algebra $\cM$ on $\XX$. For every function $f:\XX\to[0,\infty]$, the \emph{Lebesgue integral} of $f$ with respect to $\mu$ is the set function $\mu\res f:\cM\to[0,\infty]$ defined by
\[
 \mu\res f(A)=\int_Af\,d\mu:=\sup\left\{\sum_{i=1}^nc_i\,\mu(E_i\cap A)\right\}\quad\text{for all }A\in\cM,
\]
where the supremum runs over all finite lists of sets $E_1,\dots,E_n\in\cM$ and numbers $c_1,\dots,c_n\in[0,\infty)$ such that $\sum_{i=1}^nc_i\chi_{E_i}\leq f$. Here $0\cdot\infty=0$.
\end{definition}

Note that no measurability of $f$ is assumed. It follows directly from the definition that $\int_Af\,d\mu\leq\int_Ag\,d\mu$ whenever $f\leq g$. If $A\in\cM$, then $\mu\res\chi_A(B)=\mu(A\cap B)$ for all $B\in\cM$, by refining any admissible list to the finite partition generated by $E_1,\dots,E_n$. Thus, $\mu\res\chi_A$ agrees with the restriction $\mu\res A$ on $\cM$, and the two uses of the symbol $\res$ are compatible.

\begin{theorem}[the Lebesgue integral is a measure]\label{t:integral}
Let $\mu$ be a measure on a $\sigma$-algebra $\cM$ on $\XX$, and let $f:\XX\to[0,\infty]$. Then $\mu\res f$ is a measure on $\cM$, and $\mu\res f(N)=0$ whenever $\mu(N)=0$.
\end{theorem}

\begin{proof}
By Lemma \ref{l:measure-criterion}, it suffices to show that $\mu\res f(\emptyset)=0$ and that $\mu\res f$ is countably subadditive and finitely additive.

If $\mu(N)=0$, then $\sum_{i=1}^nc_i\,\mu(E_i\cap N)=0$ for every admissible list. Hence $\mu\res f(N)=0$. In particular, $\mu\res f(\emptyset)=0$.

Let $A,A_1,A_2,\dots\in\cM$ with $A\subset\bigcup_{j=1}^\infty A_j$. For every admissible list,
\[
 \sum_{i=1}^nc_i\,\mu(E_i\cap A)\leq\sum_{i=1}^nc_i\sum_{j=1}^\infty\mu(E_i\cap A_j)=\sum_{j=1}^\infty\sum_{i=1}^nc_i\,\mu(E_i\cap A_j)\leq\sum_{j=1}^\infty\mu\res f(A_j).
\]
Taking the supremum gives $\mu\res f(A)\leq\sum_{j=1}^\infty\mu\res f(A_j)$.

Let $A,B\in\cM$ be disjoint. It suffices to show that $\mu\res f(A)+\mu\res f(B)\leq\mu\res f(A\cup B)$ when the right side is finite. Let $\varepsilon>0$. Choose admissible lists with $\mu\res f(A)-\varepsilon\leq\sum_{i=1}^nc_i\,\mu(E_i\cap A)$ and $\mu\res f(B)-\varepsilon\leq\sum_{j=1}^md_j\,\mu(F_j\cap B)$. Replacing $E_i$ by $E_i\cap A$ and $F_j$ by $F_j\cap B$ keeps both lists admissible. Since $A$ and $B$ are disjoint, the combined list $\sum_ic_i\chi_{E_i\cap A}+\sum_jd_j\chi_{F_j\cap B}\leq f$ is admissible, and its value on $A\cup B$ is $\sum_ic_i\,\mu(E_i\cap A)+\sum_jd_j\,\mu(F_j\cap B)$. Hence $\mu\res f(A)+\mu\res f(B)-2\varepsilon\leq\mu\res f(A\cup B)$. Letting $\varepsilon\to0$ completes the proof.
\end{proof}

\begin{lemma}[simple functions]\label{l:simple-integral}
If $\varphi=\sum_{i=1}^n c_i\chi_{E_i}$ is a nonnegative measurable simple function, then
\[
 \int_A\varphi\,d\mu=\sum_{i=1}^n c_i\mu(E_i\cap A).
\]
\end{lemma}

\begin{proof}
The displayed sum is admissible. For the reverse inequality, refine the sets $E_i$ together with the sets in any admissible simple minorant into a finite measurable partition. On each atom the two simple functions are constant, and the minorant is bounded above by $\varphi$. Finite additivity of $\mu$ on the atoms gives the desired inequality.
\end{proof}

A function $f:\XX\to[0,\infty]$ is \emph{$\cM$ measurable} if $\{x\in\XX:f(x)>a\}\in\cM$ for all $a\in\RR$. In particular, $\{f\geq a\}\in\cM$ for every $a>0$.

\begin{theorem}[monotone convergence]\label{t:mct}
Let $\mu$ be a measure on a $\sigma$-algebra $\cM$ on $\XX$, and let $f_1\leq f_2\leq\cdots$ be functions from $\XX$ to $[0,\infty]$ with pointwise limit $f$. Then $\lim_{i\to\infty}\int_\XX f_i\,d\mu\leq\int_\XX f\,d\mu$. If, moreover, each $f_i$ is $\cM$ measurable, then equality holds.
\end{theorem}

\begin{proof}
The inequality follows from monotonicity in the integrand. Assume that each $f_i$ is $\cM$ measurable, and let $L=\lim_{i\to\infty}\int_\XX f_i\,d\mu$. Let $\varphi=\sum_{j=1}^nc_j\chi_{E_j}\leq f$ be admissible, let $0<t<1$, and put $A_i=\{x\in\XX:f_i(x)\geq t\varphi(x)\}$. Since $\varphi$ is constant on the atoms of the finite measurable partition generated by $E_1,\dots,E_n$, each $A_i$ is measurable. The sets $A_i$ increase to $\XX$: if $\varphi(x)=0$, then $x\in A_1$, and if $\varphi(x)>0$, then $t\varphi(x)<f(x)$ and $x\in A_i$ for all large $i$. Since $t\varphi\chi_{A_i}\leq f_i$, Lemma \ref{l:simple-integral} gives $t\sum_jc_j\mu(E_j\cap A_i)\leq L$ for every $i$. Continuity from below gives $t\sum_jc_j\mu(E_j)\leq L$. Taking the supremum over admissible $\varphi$ and letting $t\uparrow1$ gives $\int_\XX f\,d\mu\leq L$.
\end{proof}

\begin{lemma}[locality, null changes, and completion]\label{l:integral-locality}
Let $f,g:\XX\to[0,\infty]$ and $A\in\cM$.
\begin{enumerate}
\item If $f=g$ on $A$, then $\int_Af\,d\mu=\int_Ag\,d\mu$.
\item The same conclusion holds if $f=g$ $\mu$ almost everywhere on $A$.
\item If $(\XX,\overline{\cM},\overline\mu)$ is the completion of $(\XX,\cM,\mu)$, then $\int_A f\,d\overline\mu=\int_A f\,d\mu$.
\end{enumerate}
\end{lemma}

\begin{proof}
For (1), intersect every set in an admissible simple minorant with $A$. For (2), remove a measurable null set $N$ on which equality may fail. Theorem \ref{t:integral} gives $(\mu\res f)(N)=(\mu\res g)(N)=0$, so additivity in the set variable reduces the claim to (1).

For (3), the inequality $\int_A f\,d\mu\leq\int_A f\,d\overline\mu$ is immediate. Conversely, let $\sum_{i=1}^n c_i\chi_{B_i}\leq f$ be admissible for $\overline\mu$, with $B_i\in\overline{\cM}$. For each $i$, choose $C_i\in\cM$ such that $C_i\subset B_i$ and $\overline\mu(B_i\setminus C_i)=0$. Then $\sum_i c_i\chi_{C_i}\leq f$ and $\mu(C_i\cap A)=\overline\mu(B_i\cap A)$ for every $i$. Thus, every admissible value for the completed integral is an admissible value for the original integral. Taking the supremum proves (3).
\end{proof}

\begin{remark}[completion]\label{r:completion-convention}
We use the usual convention of suppressing the bar on a completed measure in integrals. Thus, for $A\in\overline{\cM}$, we write $\int_A f\,d\mu$ for $\int_A f\,d\overline\mu$ and $\mu(A)$ for $\overline\mu(A)$. Lemma \ref{l:integral-locality}(3) shows that this convention does not change an integral already defined on $(\XX,\cM,\mu)$.
\end{remark}

\begin{example}[measurability is necessary]\label{ex:mct-necessary}
Let $\cM=\{\emptyset,\ZZ_+\}$, let $\mu(\emptyset)=0$ and $\mu(\ZZ_+)=1$, and let $g:\ZZ_+\to[0,\infty]$ vanish at some point. If $c\chi_E\leq g$ with $E\in\cM$ and $c>0$, then $E=\emptyset$. Hence $\int_{\ZZ_+}g\,d\mu=0$. In particular, the functions $f_i=\chi_{\{1,\dots,i\}}$ increase to $f=\chi_{\ZZ_+}$, while $\int_{\ZZ_+}f_i\,d\mu=0$ for all $i$ and $\int_{\ZZ_+}f\,d\mu=1$. The sets $\{1,\dots,i\}$ are not in $\cM$, and the second assertion of Theorem \ref{t:mct} does not apply. Similarly, $\chi_{\{1\}}$ and $\chi_{\ZZ_+\setminus\{1\}}$ have integral $0$, while their sum $f$ has integral $1$.
\end{example}

\begin{remark}[measurability, additivity, and convergence]\label{r:convergence}
If $f,g:\XX\to[0,\infty]$ are measurable, then $\int_\XX(f+g)\,d\mu=\int_\XX f\,d\mu+\int_\XX g\,d\mu$. This follows by approximating $f$ and $g$ from below by measurable simple functions, using Lemma \ref{l:simple-integral}, and applying monotone convergence. Iteration and monotone convergence give $\int_\XX\sum_if_i\,d\mu=\sum_i\int_\XX f_i\,d\mu$ for every sequence of nonnegative measurable functions. For arbitrary nonnegative functions, only the inequality $\int_\XX(f+g)\,d\mu\geq\int_\XX f\,d\mu+\int_\XX g\,d\mu$ holds in general, and Example \ref{ex:mct-necessary} shows that it can be strict. Thus, Theorem \ref{t:integral} gives additivity in the set variable for every $f\geq0$, while measurability restores additivity in the integrand. The signed integral, Fatou's lemma, and the dominated convergence theorem then follow by the standard arguments.
\end{remark}

\begin{remark}[lower and upper integrals]\label{r:upper-integral}
The integral in Definition \ref{d:integral} is the lower Lebesgue integral. The upper Lebesgue integral
\[
 \overline{\int_A}f\,d\mu:=\inf\left\{\int_Ag\,d\mu:g\geq f\text{ and }g\text{ is }\cM\text{ measurable}\right\}
\]
also defines a measure on $\cM$ for every $f:\XX\to[0,\infty]$, by Lemma \ref{l:measure-criterion}. Countable subadditivity follows by patching near-optimal majorants on disjoint pieces, and finite additivity follows from Theorem \ref{t:integral} applied to each majorant.
\end{remark}

\subsection{Product measures and Tonelli's theorem}\label{ss:tonelli}

Let $(X,\mathcal M,\mu)$ and $(Y,\mathcal N,\nu)$ be measure spaces. A \emph{measurable rectangle} is a set $A\times B$ with $A\in\mathcal M$ and $B\in\mathcal N$. Let $\mathcal R$ consist of the empty set and all measurable rectangles, and let the \emph{product $\sigma$-algebra} $\mathcal M\otimes\mathcal N$ be the $\sigma$-algebra on $X\times Y$ generated by $\mathcal R$. For $E\subset X\times Y$, the sections of $E$ are $E_x=\{y\in Y:(x,y)\in E\}$ and $E^y=\{x\in X:(x,y)\in E\}$. On $\mathcal R$, define $w(A\times B)=\mu(A)\nu(B)$, where $0\cdot\infty=0$, and let $\mu\times\nu$ denote the Method I outer
measure generated by $w$.

\begin{theorem}[product measure by Method I]\label{t:product-measure}
Every measurable rectangle $A\times B$ is $\mu\times\nu$ measurable and satisfies $(\mu\times\nu)(A\times B)=\mu(A)\nu(B)$. Consequently, $\mathcal M\otimes\mathcal N\subset\cM_{\mu\times\nu}$, and
the restriction of the outer measure $\mu\times\nu$ to
$\mathcal M\otimes\mathcal N$ is the largest measure on
$\mathcal M\otimes\mathcal N$ that assigns $\mu(A)\nu(B)$ to every
measurable rectangle $A\times B$.
\end{theorem}

\begin{proof}
Suppose that $A\times B\subset\bigcup_iA_i\times B_i$ with
$A_i\times B_i\in\mathcal R$. For every $x\in A$, the sets $B_i$ with
$x\in A_i$ cover $B$, and hence $\nu(B)\leq\sum_i\chi_{A_i}(x)\nu(B_i)$.
Integrating over $A$ and using Remark \ref{r:convergence} gives
\[
 \mu(A)\nu(B)
 \leq\sum_{i=1}^\infty\mu(A\cap A_i)\nu(B_i)
 \leq\sum_{i=1}^\infty\mu(A_i)\nu(B_i).
\]
Thus, $w$ is countably subadditive on $\mathcal R$, and Proposition
\ref{p:method-one}(E) gives the asserted values.

Let $R=A\times B$ and $R'=A'\times B'$. The sets $R'\cap R$,
$(A'\setminus A)\times B'$, and $(A'\cap A)\times(B'\setminus B)$ partition
$R'$, and their weights add to $w(R')$. Proposition \ref{p:method-one}(M)
shows that $R$ is $\mu\times\nu$ measurable, and hence
$\mathcal M\otimes\mathcal N\subset\cM_{\mu\times\nu}$. Finally, let
$\lambda$ be a measure on $\mathcal M\otimes\mathcal N$ with
$\lambda(A\times B)=\mu(A)\nu(B)$ for all measurable rectangles. As in Remark
\ref{r:extension}, countable subadditivity of $\lambda$ gives
$\lambda\leq\mu\times\nu$ on $\mathcal M\otimes\mathcal N$.
\end{proof}

We write $\mu\otimes\nu$ for the restriction of the outer measure
$\mu\times\nu$ to $\mathcal M\otimes\mathcal N$. Thus, $\mu\times\nu$ is
defined on every subset of $X\times Y$, while $\mu\otimes\nu$ is a measure on
the product $\sigma$-algebra. If $\mu$ and $\nu$ are $\sigma$-finite, then
$\mu\otimes\nu$ is the unique measure on $\mathcal M\otimes\mathcal N$ that assigns
$\mu(A)\nu(B)$ to every measurable rectangle $A\times B$. Indeed, the proof of Theorem \ref{t:product-measure} shows
that $w$ satisfies the hypotheses of Proposition \ref{p:method-one}(E) and
(M), and $w$ is $\sigma$-finite. Lemma \ref{l:uniqueness} applies.

Conversely, $\mu\times\nu$ is the Method I outer measure generated by
$\mu\otimes\nu$. Since $X\times Y\in\mathcal R$, Proposition
\ref{p:method-one}(OR) gives every set $S\subset X\times Y$ a hull
$B\in\mathcal R_{\sigma\delta}\subset\mathcal M\otimes\mathcal N$ with
$(\mu\times\nu)(B)=(\mu\times\nu)(S)$. The Method I outer measure generated by
$\mu\otimes\nu$ is therefore at most $\mu\times\nu$, and it is at least
$\mu\times\nu$ by Remark \ref{r:extension}. Under $\sigma$-finiteness,
Theorem \ref{t:measure-outer-correspondence}(5) identifies
$\cM_{\mu\times\nu}$ with the completion of $\mathcal M\otimes\mathcal N$.
The completion can be strictly larger, even when $\mu$ and $\nu$ are
complete. For example, let $\mu=\nu$ be Lebesgue measure on the Lebesgue
measurable subsets of $\RR$, and let $N\subset\RR$ be a set that is not
Lebesgue measurable. Then $\{0\}\times N$ is $\mu\times\nu$ null, but it does
not belong to $\mathcal M\otimes\mathcal N$, since every section of a set in
$\mathcal M\otimes\mathcal N$ is measurable. Theorems \ref{t:tonelli-sets}
and \ref{t:tonelli} are stated for $\cM_{\mu\times\nu}$.

\begin{theorem}[Tonelli for sets]\label{t:tonelli-sets}
Assume that $\mu$ and $\nu$ are $\sigma$-finite, and let
$E\subset X\times Y$ be $\mu\times\nu$ measurable.  Then, for $\mu$ almost
every $x$, $E_x$ belongs to the completion of $\mathcal N$, and for $\nu$
almost every $y$, $E^y$ belongs to the completion of $\mathcal M$.  The
section functions $x\mapsto\nu(E_x)$ and $y\mapsto\mu(E^y)$ are measurable
with respect to the corresponding completed $\sigma$-algebras, and
\begin{equation}\label{tonelli-sets-symmetric}
 (\mu\times\nu)(E)
 =\int_X\nu(E_x)\,d\mu(x)
 =\int_Y\mu(E^y)\,d\nu(y),
\end{equation}
where each section function is extended by $0$ on the null set where it is
undefined, and we use the convention in Remark
\ref{r:completion-convention} for completed measures.
\end{theorem}

\begin{proof}
We prove the first equality. The second follows by interchanging the factors.
Every finite union $F$ of measurable rectangles is a finite disjoint union of
measurable rectangles. Hence $F_x\in\mathcal N$, the function
$x\mapsto\nu(F_x)$ is $\mathcal M$ measurable, and finite additivity gives
$(\mu\times\nu)(F)=\int_X\nu(F_x)\,d\mu(x)$. Every $F\in\mathcal R_\sigma$ is
an increasing union of such finite unions. Continuity from below for
$\mu\times\nu$ and $\nu$, followed by monotone convergence, extends the
identity to $\mathcal R_\sigma$.

Choose increasing sets $A_k\in\mathcal M$ and $B_k\in\mathcal N$ of finite
measure with $\bigcup_kA_k=X$ and $\bigcup_kB_k=Y$, and put
$R_k=A_k\times B_k$. Let $F\in\mathcal R_{\sigma\delta}$, and write
$F=\bigcap_jU_j$ with $U_j\in\mathcal R_\sigma$ decreasing, which is possible
because finite intersections preserve $\mathcal R_\sigma$. Fix $k$. The
functions $g_j(x)=\nu((U_j\cap R_k)_x)$ are measurable and decrease to
$g(x)=\nu((F\cap R_k)_x)$, and $g_1\leq\nu(B_k)\chi_{A_k}$ is integrable.
Since $g_1-g_j\uparrow g_1-g$, monotone convergence and additivity in the
integrand give $\int_Xg_j\,d\mu\downarrow\int_Xg\,d\mu$. The sets
$U_j\cap R_k$ decrease to $F\cap R_k$ inside a set of finite $\mu\times\nu$
measure, and continuity from above gives
$(\mu\times\nu)(F\cap R_k)=\int_Xg\,d\mu$. Letting $k\to\infty$, continuity
from below and monotone convergence give
\begin{equation}\label{tonelli-sigmadelta}
 (\mu\times\nu)(F)=\int_X\nu(F_x)\,d\mu(x)
 \quad\text{for all }F\in\mathcal R_{\sigma\delta}.
\end{equation}
In particular, every section $F_x$ is $\mathcal N$ measurable and the section
function is $\mathcal M$ measurable.

Suppose next that $E$ is $\mu\times\nu$ measurable and
$(\mu\times\nu)(E)<\infty$. By Proposition \ref{p:method-one}(IR) and (OR),
there are $C_1,C_2,N\in\mathcal R_{\sigma\delta}$ with
$C_1\setminus C_2\subset E\subset(C_1\setminus C_2)\cup N$,
$(\mu\times\nu)(C_1)=(\mu\times\nu)(E)$, and
$(\mu\times\nu)(C_2)=(\mu\times\nu)(N)=0$. By \eqref{tonelli-sigmadelta},
$\nu((C_2)_x)=\nu(N_x)=0$ for $\mu$ almost every $x$. For such $x$, the
section $E_x$ lies between $(C_1)_x\setminus(C_2)_x$ and
$((C_1)_x\setminus(C_2)_x)\cup N_x$. Hence $E_x$ belongs to the completion
of $\mathcal N$ and $\nu(E_x)=\nu((C_1)_x)$. Thus, the section function agrees
almost everywhere with the $\mathcal M$ measurable function
$x\mapsto\nu((C_1)_x)$, and it is measurable for the completion of
$\mathcal M$. Lemma \ref{l:integral-locality} and
\eqref{tonelli-sigmadelta} give
$\int_X\nu(E_x)\,d\mu(x)=(\mu\times\nu)(C_1)=(\mu\times\nu)(E)$.

For arbitrary $E$, apply the preceding case to $E\cap R_k$. Outside one $\mu$
null set, the sections $(E\cap R_k)_x$ belong to the completion of
$\mathcal N$ for every $k$ and increase to $E_x$. Continuity from below for
$\mu\times\nu$ and $\nu$, followed by monotone convergence, gives the first
equality in general. Interchanging the factors gives the remaining
assertions.
\end{proof}

\begin{theorem}[Tonelli]\label{t:tonelli}
Assume that $\mu$ and $\nu$ are $\sigma$-finite, and let
$f:X\times Y\to[0,\infty]$ be $\cM_{\mu\times\nu}$ measurable.  Then, for
$\mu$ almost every $x$, the function $y\mapsto f(x,y)$ is measurable with
respect to the completion of $\mathcal N$, and for $\nu$ almost every $y$,
the function $x\mapsto f(x,y)$ is measurable with respect to the completion
of $\mathcal M$.  The two iterated-integral functions are measurable with
respect to the corresponding completed $\sigma$-algebras, and
\begin{equation}\label{tonelli-functions}
 \int_{X\times Y}f\,d(\mu\times\nu)
 =\int_X\int_Yf(x,y)\,d\nu(y)\,d\mu(x)
 =\int_Y\int_Xf(x,y)\,d\mu(x)\,d\nu(y).
\end{equation}
\end{theorem}

\begin{proof}
Choose $\cM_{\mu\times\nu}$ measurable simple functions $0\leq f_i\uparrow f$.
Theorem \ref{t:tonelli-sets} and Lemma \ref{l:simple-integral} give
\eqref{tonelli-functions} for each $f_i$.  After discarding one null set,
$f_i(x,\cdot)$ is measurable with respect to the completion of $\mathcal N$
for every $i$ and increases to $f(x,\cdot)$ for almost every $x$.  Monotone
convergence on $Y$, followed by monotone convergence on $X$ and Lemma
\ref{l:integral-locality}, gives the first equality in
\eqref{tonelli-functions}.  Interchanging the factors gives the second.
\end{proof}

\begin{remark}[comparison with other proofs]\label{r:tonelli}
The proof of Theorem \ref{t:tonelli-sets} follows the hierarchy
$\mathcal R\to\mathcal R_\sigma\to\mathcal R_{\sigma\delta}\to\cM_{\mu\times\nu}$
already present in Method I, and the last step is an instance of (OR) and
(IR) in Proposition \ref{p:method-one}. The proof uses neither Dynkin's $\pi$-$\lambda$ theorem nor a monotone class lemma. Evans and Gariepy \cite[\S1.4]{EvansGariepy} also build the product outer measure from rectangle covers and pass through countable unions and countable intersections of rectangles before treating null and measurable sets. Fubini's theorem for integrable
functions follows from Theorem \ref{t:tonelli} by splitting into positive and
negative parts.
\end{remark}

\section{The continuity principle}\label{s:stieltjes}

For an outer measure $\mu$ on a metric space $\XX$, we say that $\mu$ is \emph{Borel} if every Borel set is $\mu$ measurable, and \emph{Borel regular} if $\mu$ is Borel and every set $S\subset\XX$ is contained in a Borel set $B$ with $\mu(B)=\mu(S)$. A metric space is \emph{proper} if every closed bounded set is compact. On a proper metric space, a \emph{Radon measure} is a Borel regular outer measure that is finite on bounded sets.

Proposition \ref{p:method-one}(E) reduces the extension problem to countable subadditivity on the original covering family. In topological spaces, the next lemma reduces that criterion one step further, to finite geometric information. Compactness reduces a countable open cover to finitely many members, while finite open and compact approximations control the error. This is the continuity principle that drives the constructions below.

\begin{lemma}[continuity principle]\label{l:continuity}
Let $\XX$ be a topological space and let $\mathcal E$, $w$, and $\mu_w$ be as in Lemma \ref{l:method-one}. Assume that $w$ satisfies the following three properties.
\begin{enumerate}
\item (finitely subadditive) If $E,E_1,\dots,E_k\in\mathcal E$ and $E\subset E_1\cup\dots\cup E_k$, then $w(E)\leq\sum_{i=1}^kw(E_i)$.
\item (finitely outer approximable by open sets) For every $E\in\mathcal E$ and $\varepsilon>0$, there is an open set $U\supset E$ such that $U\setminus E$ is covered by finitely many sets $E_1',\dots,E_l'\in\mathcal E$ with $\sum_{i=1}^lw(E_i')\leq\varepsilon$.
\item (finitely inner approximable by compact sets) For every $E\in\mathcal E$ and $\varepsilon>0$, there is a compact set $K\subset E$ such that $E\setminus K$ is covered by finitely many sets $E_1'',\dots,E_m''\in\mathcal E$ with $\sum_{i=1}^mw(E_i'')\leq\varepsilon$.
\end{enumerate}
Then $w$ is countably subadditive on $\mathcal E$, and hence $\mu_w(E)=w(E)$ for all $E\in\mathcal E$.
\end{lemma}

\begin{proof}
Let $E,E_1,E_2,\dots\in\mathcal E$ with $E\subset\bigcup_{i=1}^\infty E_i$, and let $\varepsilon>0$. By (3), choose a compact set $K\subset E$ and a finite family $\mathcal A\subset\mathcal E$ that covers $E\setminus K$ and has total weight at most $\varepsilon$. By (2), for every $i\geq1$, choose an open set $U_i\supset E_i$ and a finite family $\mathcal A_i\subset\mathcal E$ that covers $U_i\setminus E_i$ and has total weight at most $2^{-i}\varepsilon$. Since $K\subset\bigcup_{i=1}^\infty U_i$ and $K$ is compact, there is $n$ with $K\subset U_1\cup\dots\cup U_n$. Hence $E$ is covered by $E_1,\dots,E_n$ and the finitely many sets in $\mathcal A\cup\mathcal A_1\cup\dots\cup\mathcal A_n$. By (1),
\[
 w(E)\leq\sum_{i=1}^nw(E_i)+\varepsilon+\sum_{i=1}^n2^{-i}\varepsilon\leq 2\varepsilon+\sum_{i=1}^\infty w(E_i).
\]
Letting $\varepsilon\to0$ shows that $w$ is countably subadditive. Proposition \ref{p:method-one}(E) completes the proof.
\end{proof}

\begin{remark}\label{r:continuity-trivial}
If every set in $\mathcal E$ is open, then condition (2) in Lemma \ref{l:continuity} holds trivially. If every set in $\mathcal E$ is compact, then condition (3) holds trivially.
\end{remark}

\begin{example}[failure of compact inner approximation]\label{ex:continuity-failure}
In Example \ref{ex:rational-intervals}, conditions (1) and (2) of Lemma \ref{l:continuity} hold, but condition (3) fails for every $\varepsilon<b-a$. Indeed, suppose that $K\subset\QQ\cap(a,b]$ is compact and that $\QQ\cap(a,b]\setminus K$ is covered by sets $\QQ\cap(a_i,b_i]$, $1\leq i\leq m$, with $\sum_{i=1}^m(b_i-a_i)<b-a$. The open set $(a,b)\setminus\bigcup_{i=1}^m[a_i,b_i]$ is nonempty and contains an interval $(c,d)$. Then $\QQ\cap(c,d)\subset K$, and the closed subset $\QQ\cap[c',d']$ of $K$ is compact for all $c<c'<d'<d$. This is impossible, since $\QQ\cap[c',d']$ is not closed in $\RR$.
\end{example}

\begin{remark}[compact classes]\label{r:marczewski}
A family $\mathcal K$ of sets is a \emph{compact class} if every countable subfamily with empty intersection has a finite subfamily with empty intersection. The classical compact-class theorem says that a finitely additive set function on an algebra that is inner regular with respect to a compact class is countably additive; see Marczewski \cite{Marczewski} and Bogachev \cite[\S1.4]{Bogachev}. Lemma \ref{l:continuity} uses compactness in a similar way, but it needs no algebra and no finite additivity. In its proof, the compact set from condition (3) is covered by the open sets from condition (2), and compactness reduces this cover to a finite one. Finite subadditivity in condition (1) then closes the argument. The conclusion is countable subadditivity, which is the hypothesis of Proposition \ref{p:method-one}(E), rather than countable additivity.
\end{remark}

\subsection{Lebesgue and Lebesgue--Stieltjes measures on \texorpdfstring{$\RR^n$}{Rn}}\label{ss:stieltjes}

An \emph{$h$-rectangle} in $\RR^n$ is a set
$R=(a_1,b_1]\times\dots\times(a_n,b_n]$ with $a_j<b_j$ for all $j$, and
$\vol R=\prod_{j=1}^n(b_j-a_j)$.

Lebesgue measure assigns to an $h$-rectangle the product $\vol R$ of its coordinate increments. Lebesgue--Stieltjes measures replace this product formula by the full mixed difference of a function at the $2^n$ vertices. The idea of generating a measure on $\RR^n$ from the mixed differences of a single function goes back to Lebesgue \cite{Lebesgue1910}, and it appears in some books, for example Doob \cite{Doob1994} and Durrett \cite{Durrett}. Most modern treatments of measure theory construct Lebesgue--Stieltjes measures only on $\RR$ and obtain measures on $\RR^n$ as products. The Method I formulation below treats all dimensions at once, and by Theorem \ref{t:stieltjes-converse} it produces every Radon measure on $\RR^n$, including measures that are not products (Example \ref{ex:diagonal}). Lebesgue measure is the case $f(x)=x_1\cdots x_n$ (Example \ref{ex:stieltjes}). The coordinate-difference notation below records the mixed difference.

\begin{definition}[signed variation]\label{d:signed-variation}
For an $h$-rectangle $R=(a_1,b_1]\times\dots\times(a_n,b_n]$, let $V_R:=\{a_1,b_1\}\times\dots\times\{a_n,b_n\}$ be its set of vertices, and for $v=(v_1,\dots,v_n)\in V_R$ put $\operatorname{sign}(v)=+1$ if $\#\{j:v_j=a_j\}$ is even and $\operatorname{sign}(v)=-1$ otherwise. For every function $f:\RR^n\to\RR$, the \emph{signed variation} of $f$ over $R$ is
\[
 \var^\pm(f,R):=\sum_{v\in V_R}\operatorname{sign}(v)f(v).
\]
\end{definition}

When $n=1$, $\var^\pm(f,(a,b])=f(b)-f(a)$. When $n=2$ and $R=(a_1,b_1]\times(a_2,b_2]$, $\var^\pm(f,R)=f(b_1,b_2)-f(a_1,b_2)-f(b_1,a_2)+f(a_1,a_2)$. With $(a_1,\dots,a_n)$ at the lower left (front) corner, the signs at the vertices for $n=2$ and $n=3$ are
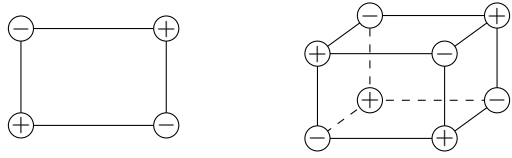
\[
\begin{tikzpicture}[baseline=(current bounding box.center),scale=0.8,
 vtx/.style={circle,draw,fill=white,inner sep=0pt,minimum size=9pt,font=\scriptsize}]
 \draw (0,0) rectangle (2.4,1.6);
 \node[vtx] at (0,0) {$+$}; \node[vtx] at (2.4,0) {$-$};
 \node[vtx] at (0,1.6) {$-$}; \node[vtx] at (2.4,1.6) {$+$};
\end{tikzpicture}
\qquad\qquad
\begin{tikzpicture}[baseline=(current bounding box.center),scale=0.7,
 x={(1cm,0cm)},y={(0.55cm,0.4cm)},z={(0cm,1cm)},
 vtx/.style={circle,draw,fill=white,inner sep=0pt,minimum size=9pt,font=\scriptsize}]
 \def\L{2.4}\def\D{1.8}\def\H{1.6}
 \draw[dashed] (0,\D,0) -- (\L,\D,0); \draw[dashed] (0,\D,0) -- (0,0,0); \draw[dashed] (0,\D,0) -- (0,\D,\H);
 \draw (0,0,0) -- (\L,0,0) -- (\L,\D,0) -- (\L,\D,\H) -- (0,\D,\H) -- (0,0,\H) -- cycle;
 \draw (0,0,\H) -- (\L,0,\H) -- (\L,\D,\H); \draw (\L,0,0) -- (\L,0,\H);
 \node[vtx] at (0,0,0) {$-$}; \node[vtx] at (\L,0,0) {$+$};
 \node[vtx] at (0,\D,0) {$+$}; \node[vtx] at (0,0,\H) {$+$};
 \node[vtx] at (\L,\D,0) {$-$}; \node[vtx] at (\L,0,\H) {$-$};
 \node[vtx] at (0,\D,\H) {$-$}; \node[vtx] at (\L,\D,\H) {$+$};
\end{tikzpicture}
\]
Crossing an edge swaps one coordinate between $a_j$ and $b_j$ and changes the sign.
Let $\Delta^j_{a_j,b_j}$ denote the difference operator in
the $j$th coordinate.  The operators commute, and
\begin{equation}\label{mixed-difference}
 \var^\pm(f,R)=
 \Delta^1_{a_1,b_1}\cdots\Delta^n_{a_n,b_n}f.
\end{equation}

\begin{lemma}[additivity]\label{l:var-additive}
If the $h$-rectangles $R_1,\dots,R_k$ partition the $h$-rectangle $R$, then
\[
 \var^\pm(f,R)=\sum_{i=1}^k\var^\pm(f,R_i).
\]
\end{lemma}

\begin{proof}
\emph{Special case.} Suppose that $k=2$ and that $R$ is split by the hyperplane $x_j=c$, where $a_j<c<b_j$. That is, $R_1$ and $R_2$ agree with $R$ except that their $j$th sides are $(a_j,c]$ and $(c,b_j]$. In the $j$th coordinate, the telescoping identity $g(b_j)-g(a_j)=\bigl(g(b_j)-g(c)\bigr)+\bigl(g(c)-g(a_j)\bigr)$ reads $\Delta^j_{a_j,b_j}=\Delta^j_{a_j,c}+\Delta^j_{c,b_j}$. The difference operators are linear, so by \eqref{mixed-difference},
\begin{equation*}
\begin{split}
 \var^\pm(f,R)
 &=\Delta^1_{a_1,b_1}\cdots\bigl(\Delta^j_{a_j,c}+\Delta^j_{c,b_j}\bigr)\cdots\Delta^n_{a_n,b_n}f\\
 &=\var^\pm(f,R_1)+\var^\pm(f,R_2).
\end{split}
\end{equation*}

\emph{General case.} Simply apply the special case a finite number of times. For each $j$, list the $j$th endpoints of $R_1,\dots,R_k$ in increasing order. These points include $a_j$ and $b_j$, and they cut $R$ into an $m_1\times\cdots\times m_n$ grid of cells. Every cell lies in exactly one $R_i$, and each $R_i$ is the union of the cells it contains. Cutting a rectangle along one hyperplane $x_j=c$ at a time, the special case shows that $\var^\pm(f,R)$ is the sum of $\var^\pm(f,Q)$ over all cells $Q$, and that $\var^\pm(f,R_i)$ is the sum of $\var^\pm(f,Q)$ over the cells $Q\subset R_i$. Summing the latter identity over $i$ gives the former. It is helpful to draw a picture.
\end{proof}

\begin{definition}[Lebesgue--Stieltjes measures]\label{d:stieltjes}
Let $f:\RR^n\to\RR$ satisfy
\begin{enumerate}
\item (nonnegative signed variation) $\var^\pm(f,R)\geq0$ for every $h$-rectangle $R$, and
\item (right continuity) $f(x+h)\to f(x)$ as $h\to0$ with
$h\in[0,\infty)^n$, for every $x\in\RR^n$.
\end{enumerate}
The \emph{Lebesgue--Stieltjes measure} $\mu_f$ is the Method I outer measure on $\RR^n$ with $\mathcal E$ the family of $h$-rectangles and the empty set and with weight $w(R)=\var^\pm(f,R)$.
\end{definition}

\begin{theorem}[Lebesgue--Stieltjes construction]\label{t:stieltjes}
Let $f$ be as in Definition \ref{d:stieltjes}. Then $\mu_f$ is a Radon measure on $\RR^n$ and $\mu_f(R)=\var^\pm(f,R)$ for every $h$-rectangle $R$.
\end{theorem}

\begin{proof}
Write $w(R)=\var^\pm(f,R)$. Let $R,R_1,\dots,R_k$ be $h$-rectangles with
$R\subset R_1\cup\dots\cup R_k$. A common grid refinement partitions each of
these rectangles into grid cells. Every cell contained in $R$ is contained in
or disjoint from each $R_i$, and hence lies in some $R_i$. Lemma
\ref{l:var-additive} and the nonnegativity of $w$ on the cells give
$w(R)\leq\sum_iw(R_i)$. This proves condition (1) of Lemma
\ref{l:continuity}.

Fix $R=(a,b]$ and $\varepsilon>0$.  If
$R^+_\delta=(a,b+\delta\mathbf 1]$, then right continuity at the finitely
many vertices of $R$ gives $w(R^+_\delta)\to w(R)$ as
$\delta\downarrow0$.  Since $R^+_\delta$ is the disjoint union of $R$ and
finitely many $h$-rectangles, Lemma \ref{l:var-additive} and nonnegativity
show that those remaining rectangles have total weight
$w(R^+_\delta)-w(R)$.  For small $\delta$, this is at most $\varepsilon$,
and they cover
\[
 \prod_{j=1}^n(a_j,b_j+\delta)\setminus R.
\]
This proves condition (2) of Lemma \ref{l:continuity}.

For sufficiently small $\delta>0$, let
$R^-_\delta=(a+\delta\mathbf 1,b]$ and
$K_\delta=\prod_{j=1}^n[a_j+\delta,b_j]$.  Again,
$w(R^-_\delta)\to w(R)$ by right continuity.  The compact set $K_\delta$ is
contained in $R$, and $R\setminus K_\delta$ is covered by the rectangular
pieces of $R\setminus R^-_\delta$, whose total weight is
$w(R)-w(R^-_\delta)$.  This proves condition (3).  Lemma \ref{l:continuity}
now gives $\mu_f(R)=w(R)$.

For two $h$-rectangles $R$ and $R'$, the intersection $R'\cap R$ is an
$h$-rectangle or empty, while $R'\setminus R$ is a finite disjoint union of
$h$-rectangles.  Lemma \ref{l:var-additive} verifies the hypothesis of
Proposition \ref{p:method-one}(M), and $h$-rectangles are $\mu_f$ measurable.
Every open subset of $\RR^n$ is a countable union of $h$-rectangles. Hence
every Borel set is $\mu_f$ measurable.  Proposition
\ref{p:method-one}(OR) gives Borel regularity.  Finally, every bounded set lies
in an $h$-rectangle of finite weight because $f$ is real valued.  Therefore,
$\mu_f$ is a Radon measure.
\end{proof}

\begin{theorem}[Radon measures and Lebesgue--Stieltjes measures]\label{t:stieltjes-converse}
An outer measure $\mu$ on $\RR^n$ is a Radon measure if and only if
$\mu=\mu_f$ for some function $f$ satisfying Definition \ref{d:stieltjes}.
\end{theorem}

\begin{proof}
If $\mu=\mu_f$ for a function in Definition \ref{d:stieltjes}, then $\mu$ is
Radon by Theorem \ref{t:stieltjes}.  Conversely, let $\mu$ be a Radon measure.
For $t\in\RR$, set
\[
 I(t)=(\min\{0,t\},\max\{0,t\}],\qquad
 \epsilon(t)=\begin{cases}1,&t\geq0,\\-1,&t<0.\end{cases}
\]
Thus, $I(0)=\emptyset$.  Define a representing function based at the origin by
\begin{equation}\label{canonical-representative}
 f(x)=\left(\prod_{j=1}^n\epsilon(x_j)\right)
 \mu\left(\prod_{j=1}^n I(x_j)\right).
\end{equation}
The product rectangle is bounded, so $f$ is real valued.

For a one-dimensional variable $s$, put
\[
 g_t(s)=\begin{cases}
  \chi_{(0,t]}(s),&t\geq0,\\
  -\chi_{(t,0]}(s),&t<0.
 \end{cases}
\]
The identity $g_b-g_a=\chi_{(a,b]}$ holds whenever $a<b$.  Expanding one
difference in each coordinate, the alternating sum in
\eqref{mixed-difference} is therefore the same finite inclusion--exclusion
sum obtained from
\[
 \prod_{j=1}^n(g_{b_j}-g_{a_j})=\chi_{(a,b]}.
\]
Partition the bounded union of the rectangles in
\eqref{canonical-representative} into finitely many rectangles on which all
of the functions $g_{a_j}$ and $g_{b_j}$ are constant.  Finite additivity of
$\mu$ on this partition gives
\begin{equation}\label{representative-difference}
 \var^\pm(f,(a,b])=\mu((a,b]).
\end{equation}
In particular, the signed variation is nonnegative.

We next check right continuity. Fix $x\in\RR^n$, and let $0<\eta<1$ be
smaller than $|x_j|$ for every $j$ with $x_j\neq0$. For $h\in[0,\eta)^n$ and
every $j$,
\[
 \epsilon(x_j+h_j)=\epsilon(x_j)
 \qquad\text{and}\qquad
 I(x_j+h_j)\triangle I(x_j)=(x_j,x_j+h_j].
\]
Hence the rectangles in \eqref{canonical-representative} for $x$ and $x+h$
lie in the bounded rectangle $B=\prod_{j=1}^n(-|x_j|-1,|x_j|+1]$, and their
symmetric difference is contained in the union of the sets
$S_j(\eta):=\{s\in B:x_j<s_j\leq x_j+\eta\}$ with $1\leq j\leq n$. Thus,
\[
 |f(x+h)-f(x)|\leq\sum_{j=1}^n\mu(S_j(\eta))\quad\text{for all }h\in[0,\eta)^n.
\]
Each $S_j(\eta)$ decreases to $\emptyset$ as $\eta\downarrow0$ inside the set
$B$ of finite measure. Continuity from above gives $f(x+h)\to f(x)$ as $h\to0$
through $[0,\infty)^n$.

Hence $f$ satisfies Definition \ref{d:stieltjes}.  Theorem
\ref{t:stieltjes} and \eqref{representative-difference} show that $\mu_f$ and
$\mu$ agree on every $h$-rectangle.  The family of $h$-rectangles and the
empty set satisfies Proposition \ref{p:method-one}(M), the weight
$\var^\pm(f,\cdot)$ is countably subadditive by Theorem \ref{t:stieltjes}, and
$\RR^n$ is covered by countably many $h$-rectangles of finite weight.  Lemma
\ref{l:uniqueness} therefore gives equality on the Borel $\sigma$-algebra.
Both outer measures are Borel regular, so they agree on every subset of
$\RR^n$.
\end{proof}

\begin{example}\label{ex:stieltjes}
Let $g_j:\RR\to\RR$ be right-continuous and nondecreasing for $1\leq j\leq n$, and let $f(x)=\prod_jg_j(x_j)$. Then
\[
 \var^\pm(f,(a,b])=\prod_{j=1}^n\bigl(g_j(b_j)-g_j(a_j)\bigr),
\]
and $\mu_f$ represents the product of the one-dimensional Lebesgue--Stieltjes measures. Taking $g_j(t)=t$ for every $j$ gives $\var^\pm(f,R)=\vol R$. The measure $\Leb^n:=\mu_f$ is \emph{Lebesgue measure} on $\RR^n$, and Theorem \ref{t:stieltjes} shows that it is a Radon measure with $\Leb^n(R)=\vol R$ for every $h$-rectangle $R$. Taking every $g_j$ to be $\chi_{[0,\infty)}$ gives the Dirac mass at the origin, and taking every $g_j$ to be the Cantor distribution function gives the product Cantor measure.
\end{example}

\begin{example}[a measure on the diagonal]\label{ex:diagonal}
Let $n=2$ and $f(x_1,x_2)=\min\{x_1,x_2\}$. Let $\nu$ be the image of Lebesgue measure on $\RR$ under $t\mapsto(t,t)$, i.e.~$\nu(S)=\Leb^1(\{t\in\RR:(t,t)\in S\})$ for all $S\subset\RR^2$. Since $\Leb^1$ is Radon, so is $\nu$. Given an $h$-rectangle $R=(a_1,b_1]\times(a_2,b_2]$, choose $c<\min\{a_1,a_2\}$. For every vertex $v$ of $R$, $f(v)=c+\nu((c,v_1]\times(c,v_2])$. Since the signed variation of a constant vanishes, inclusion-exclusion gives
\[
 \var^\pm(f,R)=\nu(R)=\Leb^1((a_1,b_1]\cap(a_2,b_2]).
\]
The function $f$ is continuous, and it satisfies the hypotheses of Definition \ref{d:stieltjes}. As in the proof of Theorem \ref{t:stieltjes-converse}, Lemma \ref{l:uniqueness} and Borel regularity give $\mu_f=\nu$. Thus, $\mu_f$ is carried by the diagonal. In particular, $\mu_f$ is not a product of measures on $\RR$.
\end{example}

\subsection{Riesz representation}\label{ss:riesz}

The continuity principle does not require a metric. A classical instance is the Riesz representation theorem. Its usual proof defines a weight on open sets from a linear functional, proves countable subadditivity by compactness and partitions of unity, and then shows that open sets are measurable. Below, the first step is an application of Lemma \ref{l:continuity}, and the second is the measurability criterion in Remark \ref{r:measurability}. We state the version for functionals on vector-valued functions, which produces the total variation measure and the polar decomposition used in geometric measure theory. See Evans and Gariepy \cite[\S1.8]{EvansGariepy}. Steps that follow the standard proof are only indicated.

Let $\XX$ be a locally compact Hausdorff space in which every open set is $\sigma$-compact. This holds for $\RR^n$, for its open subsets, and for every proper metric space. Fix $d\geq1$, and write $|\cdot|$ for the Euclidean norm on $\RR^d$. For open $V\subset\XX$, let $C_c(V;\RR^d)$ denote the set of continuous functions $f:\XX\to\RR^d$ whose support is a compact subset of $V$. Let $\mathcal E$ be the family of open sets with compact closure. Since $\XX$ is locally compact, every compact set is contained in a set in $\mathcal E$.

\begin{theorem}[Riesz representation]\label{t:riesz}
Let $L:C_c(\XX;\RR^d)\to\RR$ be linear, and assume that
\begin{equation}\label{riesz-bounded}
 \sup\{L(f):f\in C_c(\XX;\RR^d),\ |f|\leq1,\ \operatorname{spt}f\subset K\}<\infty\quad\text{for every compact }K\subset\XX.
\end{equation}
For open $V\subset\XX$, put
\[
 w(V):=\sup\{L(f):f\in C_c(V;\RR^d),\ |f|\leq1\},
\]
and let $\mu$ be the Method I outer measure generated by $w$ on $\mathcal E$. Then the following hold.
\begin{enumerate}
\item $\mu(V)=w(V)$ for every open set $V$, $\mu(A)=\inf\{\mu(V):V\supset A\text{ open}\}$ for every $A\subset\XX$, and $\mu(K)<\infty$ for every compact $K$.
\item Every Borel set is $\mu$ measurable, and $\mu(V)=\sup\{\mu(K):K\subset V\text{ compact}\}$ for every open set $V$.
\item There is a $\mu$ measurable function $\sigma:\XX\to\RR^d$ with $|\sigma|=1$ $\mu$-a.e.\ such that
\[
 L(f)=\int_\XX f\cdot\sigma\,d\mu\quad\text{for all }f\in C_c(\XX;\RR^d).
\]
\end{enumerate}
If $d=1$ and $L(f)\geq0$ whenever $f\geq0$, then $\sigma=1$ $\mu$-a.e., and part (3) reads $L(f)=\int_\XX f\,d\mu$.
\end{theorem}

\begin{proof}[Sketch of proof]
By \eqref{riesz-bounded}, $w$ is finite on $\mathcal E$. Functions with disjoint supports can be added without increasing the norm. Thus, if $E\in\mathcal E$ and $f\in C_c(E;\RR^d)$ with $|f|\leq1$, then $f+g\in C_c(E;\RR^d)$ and $|f+g|\leq1$ for every $g\in C_c(E\setminus\operatorname{spt}f;\RR^d)$ with $|g|\leq1$. Taking the supremum over $g$ gives
\begin{equation}\label{riesz-disjoint}
 L(f)+w(E\setminus\operatorname{spt}f)\leq w(E).
\end{equation}

We verify the three conditions of Lemma \ref{l:continuity} for $w$ on $\mathcal E$. For condition (1) of Lemma \ref{l:continuity}, let $E\subset E_1\cup\dots\cup E_k$ with $E,E_1,\dots,E_k\in\mathcal E$, and let $f\in C_c(E;\RR^d)$ with $|f|\leq1$. Choose a partition of unity $h_1,\dots,h_k$ on $\operatorname{spt}f$ with $h_i\in C_c(E_i)$, $0\leq h_i$, $\sum_ih_i\leq1$, and $\sum_ih_i=1$ on $\operatorname{spt}f$. Then $f=\sum_ih_if$ with $h_if\in C_c(E_i;\RR^d)$ and $|h_if|\leq1$. Hence $L(f)\leq\sum_iw(E_i)$. Condition (2) of Lemma \ref{l:continuity} holds with $U=E$. For condition (3) of Lemma \ref{l:continuity}, given $E\in\mathcal E$ and $\varepsilon>0$, choose $f\in C_c(E;\RR^d)$ with $|f|\leq1$ and $L(f)>w(E)-\varepsilon$, and put $K=\operatorname{spt}f$. Then $E\setminus K\in\mathcal E$, and \eqref{riesz-disjoint} gives $w(E\setminus K)<\varepsilon$. By Lemma \ref{l:continuity}, $\mu(E)=w(E)$ for every $E\in\mathcal E$.

Let $A\subset\XX$ be open, let $E\in\mathcal E$, and let $\varepsilon>0$. Choose $f\in C_c(E\cap A;\RR^d)$ with $|f|\leq1$ and $L(f)>w(E\cap A)-\varepsilon$. The set $E\setminus\operatorname{spt}f$ belongs to $\mathcal E$ and contains $E\setminus A$. Since $E\cap A\in\mathcal E$, \eqref{riesz-disjoint} gives
\[
 \mu(E\cap A)+\mu(E\setminus A)
 \leq w(E\cap A)+w(E\setminus\operatorname{spt}f)
 < L(f)+\varepsilon+w(E)-L(f)=w(E)+\varepsilon.
\]
By Remark \ref{r:measurability}, $A$ is $\mu$ measurable. Hence every Borel set is $\mu$ measurable.

Let $V$ be open. Since $V$ is $\sigma$-compact and $\XX$ is locally compact, $V$ is the union of an increasing sequence $E_1\subset E_2\subset\cdots$ of sets in $\mathcal E$. Every $f\in C_c(V;\RR^d)$ has support in some $E_j$. Continuity from below gives $\mu(V)=\lim_j\mu(E_j)=\lim_jw(E_j)=w(V)$. Countable unions of sets in $\mathcal E$ are open, and Proposition \ref{p:method-one}(OA) gives outer regularity. Every compact set lies in a set $E\in\mathcal E$, and it has measure at most $w(E)<\infty$. If $f\in C_c(V;\RR^d)$ with $|f|\leq1$, then every open set containing $\operatorname{spt}f$ has measure at least $L(f)$, and outer regularity gives $\mu(\operatorname{spt}f)\geq L(f)$. This proves inner regularity on open sets and completes parts (1) and (2) of the theorem.

The remaining steps follow the standard proof \cite[\S1.8]{EvansGariepy}. Writing $f\in C_c(\XX;\RR^d)$ as a finite sum of functions supported in the superlevel sets $\{|f|>t\}$ and applying part (1) of the theorem to these open sets gives
\[
 |L(f)|\leq\int_\XX|f|\,d\mu\quad\text{for all }f\in C_c(\XX;\RR^d).
\]
The measure $\mu$ is $\sigma$-finite, and $C_c(\XX;\RR^d)$ is dense in $L^1(\mu;\RR^d)$ by parts (1) and (2). The duality between $L^1(\mu;\RR^d)$ and $L^\infty(\mu;\RR^d)$ gives $\sigma$ with $|\sigma|\leq1$ $\mu$-a.e.\ and $L(f)=\int f\cdot\sigma\,d\mu$. For $E\in\mathcal E$,
\[
 \mu(E)=w(E)=\sup\left\{\int_\XX f\cdot\sigma\,d\mu:f\in C_c(E;\RR^d),\ |f|\leq1\right\}\leq\int_E|\sigma|\,d\mu,
\]
and $\mu(E)<\infty$ forces $|\sigma|=1$ $\mu$-a.e.\ on $E$. The sets in $\mathcal E$ cover $\XX$ by countably many members, which proves part (3). If $d=1$ and $L$ is positive, then $\int f\sigma\,d\mu\geq0$ for all $f\geq0$, and $\sigma=1$ $\mu$-a.e.
\end{proof}

\section{Kolmogorov extension}\label{s:kolmogorov}

Kolmogorov's extension theorem is the basic existence theorem for probability measures specified by finite-dimensional distributions. We first set up the notation.

Let $I$ be a nonempty index set, and let $S$ be a metric space. Below, $S$ will be either $\RR$ or a compact metric space. For every finite set $F\subset I$, write
\[
 \pi_F:S^I\to S^F,\qquad \pi_F(x)=x|_F,
\]
for the coordinate projection. If $F\subset G\subset I$ are finite, write
\[
 \pi_F^G:S^G\to S^F,\qquad \pi_F^G(x)=x|_F,
\]
for the projection that forgets the coordinates in $G\setminus F$. Thus,
$\pi_F=\pi_F^G\circ\pi_G$. For a Borel set $B\subset S^F$, the set
$\pi_F^{-1}(B)$ is called a \emph{cylinder determined by $F$}. Membership in
it depends only on the coordinates indexed by $F$. If $F\subset G$, then
\begin{equation}\label{cylinder-lift}
 \pi_F^{-1}(B)
 =\pi_G^{-1}\bigl((\pi_F^G)^{-1}(B)\bigr).
\end{equation}
The \emph{product $\sigma$-algebra} on $S^I$ is the $\sigma$-algebra
generated by the cylinders, equivalently the smallest $\sigma$-algebra for
which all finite-coordinate projections $\pi_F$ are measurable.

A family $(\mu_F)$ of Borel probability measures on $S^F$, indexed by the finite sets $F\subset I$, is \emph{consistent} if
\begin{equation}\label{consistent-family}
 \mu_G\bigl((\pi_F^G)^{-1}(B)\bigr)=\mu_F(B)
\end{equation}
whenever $F$ and $G$ are finite subsets of $I$ with $F\subset G$ and $B\subset S^F$ is Borel. In words, forgetting coordinates from a $G$-dimensional distribution recovers the prescribed $F$-dimensional distribution.

For example, let $\rho$ be a Borel probability measure on $\RR$, and for finite $F\subset I$ let $\mu_F=\bigotimes_{i\in F}\rho$ be the product of copies of $\rho$, one for each coordinate indexed by $F$. In the notation of Section \ref{ss:tonelli}, $\mu_F$ is obtained by iterating $\otimes$. Since the Borel $\sigma$-algebra of $\RR^F$ is the product of the Borel $\sigma$-algebras of the factors, Lemma \ref{l:uniqueness} shows that $\mu_F$ is the unique Borel probability measure on $\RR^F$ with
\begin{equation}\label{product-law}
 \mu_F\Bigl(\prod_{i\in F}B_i\Bigr)=\prod_{i\in F}\rho(B_i)\quad\text{for all Borel sets }B_i\subset\RR.
\end{equation}
Forgetting the coordinates in $G\setminus F$ multiplies the right side of \eqref{product-law} by factors $\rho(\RR)=1$, and Lemma \ref{l:uniqueness} shows that the family $(\mu_F)$ is consistent. When $I=\mathbb N$, Theorem \ref{t:kolmogorov} below produces a probability measure on $\RR^{\mathbb N}$ under which the coordinate maps $x\mapsto x_i$ are independent and have common law $\rho$.

In the extension-first organization of this paper, this is again an extension problem. Consistency prescribes a weight on cylinders, and the task is to prove that the Method I outer measure retains those weights. When $S$ is compact, the continuity principle verifies the required countable subadditivity (Lemma \ref{l:kolmogorov-compact}), and no premeasure on an algebra of cylinders is needed. Together with the uniqueness argument in the proof of Theorem \ref{t:kolmogorov}, this gives the theorem with $\RR$ replaced by any compact metric space. The case $S=\RR$, which is needed for sequences of independent random variables with an arbitrary common law on $\RR$, follows by mapping $\RR$ homeomorphically onto $(0,1)\subset[0,1]$ and using the outer measure to show that $(0,1)^I$ has full measure in $[0,1]^I$ (Lemma \ref{l:full-outer-measure}). As in Folland \cite[Chapter 10]{Folland}, we first treat compact factors and then pass to $\RR$ through a compactification. Folland handles the compact case with the Riesz representation theorem (compare Nelson \cite{Nelson1964}). Here, Lemma \ref{l:kolmogorov-compact} obtains it from Method I and the continuity principle, and Lemma \ref{l:full-outer-measure} returns to $\RR$ through a set that need not be measurable.

\begin{lemma}[compact factors]\label{l:kolmogorov-compact}
Let $S$ be a compact metric space, and let $(\nu_F)$ be a consistent family of Borel probability measures on the spaces $S^F$. Let $\mathcal E$ consist of the empty set and all cylinders in $S^I$, and define
\begin{equation}\label{kolmogorov-weight}
 w\bigl(\pi_F^{-1}(B)\bigr):=\nu_F(B).
\end{equation}
Then $w$ is well defined, the Method I outer measure $\mu_w$ satisfies $\mu_w=w$ on $\mathcal E$, every cylinder is $\mu_w$ measurable, and $\mu_w$ restricts to a probability measure on the product $\sigma$-algebra of $S^I$.
\end{lemma}

\begin{proof}
Suppose that $\pi_F^{-1}(B)=\pi_G^{-1}(D)$, and put $H=F\cup G$. Viewed in $S^H$, the same equality reads $(\pi_F^H)^{-1}(B)=(\pi_G^H)^{-1}(D)$. Consistency gives
\[
 \nu_F(B)
 =\nu_H\bigl((\pi_F^H)^{-1}(B)\bigr)
 =\nu_H\bigl((\pi_G^H)^{-1}(D)\bigr)
 =\nu_G(D).
\]
Thus, $w$ is well defined.

We verify the hypotheses of Lemma \ref{l:continuity} in $S^I$ with the product topology. By Tychonoff's theorem, $S^I$ is compact. For finite subadditivity, suppose that $\pi_F^{-1}(B)$ is covered by $\pi_{F_1}^{-1}(B_1),\dots,\pi_{F_m}^{-1}(B_m)$, and put $H=F\cup F_1\cup\dots\cup F_m$. By \eqref{cylinder-lift}, each of these cylinders is $\pi_H^{-1}$ of a Borel set in $S^H$, and by consistency its weight is the $\nu_H$ measure of that set. The corresponding sets in $S^H$ satisfy the same covering relation, because every point of $S^H$ is $\pi_H(y)$ for some $y\in S^I$. The desired inequality is therefore finite subadditivity of $\nu_H$.

For the approximation conditions, recall that a finite Borel measure on a metric space is outer regular by open sets and inner regular by closed sets. See Billingsley \cite[Theorem 1.1]{Billingsley1999} and Kechris \cite[Theorem 17.10]{Kechris}, where the proof shows that the Borel sets with both approximations form a $\sigma$-algebra containing the closed sets. Since $S^F$ is compact, closed subsets of $S^F$ are compact.

Given $\varepsilon>0$, choose an open set $V\supset B$ and a compact set $K\subset B$ in $S^F$ with $\nu_F(V\setminus B)<\varepsilon$ and $\nu_F(B\setminus K)<\varepsilon$. Then $\pi_F^{-1}(V)$ is open in $S^I$ and contains $\pi_F^{-1}(B)$, and $\pi_F^{-1}(V)\setminus\pi_F^{-1}(B)=\pi_F^{-1}(V\setminus B)$ has weight less than $\varepsilon$. The cylinder $\pi_F^{-1}(K)$ is closed in the compact space $S^I$, hence compact, and $\pi_F^{-1}(B)\setminus\pi_F^{-1}(K)$ has weight less than $\varepsilon$. Lemma \ref{l:continuity} gives $\mu_w=w$ on $\mathcal E$.

Let $\pi_F^{-1}(B),\pi_G^{-1}(D)\in\mathcal E$, and put $H=F\cup G$. In $S^H$, the sets corresponding to their intersection and difference are Borel. Additivity of $\nu_H$ therefore verifies the hypothesis of Proposition \ref{p:method-one}(M), and every cylinder is $\mu_w$ measurable. Hence the product $\sigma$-algebra of $S^I$ is contained in $\cM_{\mu_w}$. Since $S^I=\pi_{\{i\}}^{-1}(S)$ for any $i\in I$, we have $\mu_w(S^I)=1$.
\end{proof}

\begin{lemma}[product sets of full outer measure]\label{l:full-outer-measure}
In the setting of Lemma \ref{l:kolmogorov-compact}, let $U\subset S$ be a Borel set with $\nu_{\{i\}}(U)=1$ for every $i\in I$, and put $Z=U^I$. Then $\mu_w(Z)=1$. Moreover, the set function
\[
 A\cap Z\longmapsto\mu_w(A)\quad\text{for all }A\text{ in the product }\sigma\text{-algebra of }S^I
\]
is a well-defined probability measure on the trace of that $\sigma$-algebra on $Z$.
\end{lemma}

\begin{proof}
Consider any countable cylinder cover $Z\subset\bigcup_j\pi_{F_j}^{-1}(B_j)$, and put $J=\bigcup_jF_j$. The set
\[
 W_J:=\{y\in S^I:y_i\in U\text{ for every }i\in J\}
\]
is a countable intersection of the cylinders $\pi_{\{i\}}^{-1}(U)$ with $i\in J$, each of measure $1$. Hence $\mu_w(W_J)=1$. Moreover, $W_J$ is contained in the given cover. Indeed, if $y\in W_J$, extend $y|_J$ to a point $z\in Z$, which is possible because $U\neq\emptyset$. Some $\pi_{F_j}^{-1}(B_j)$ contains $z$. Since that cylinder depends only on coordinates in $F_j\subset J$, it also contains $y$. Therefore, every countable cylinder cover of $Z$ has total weight at least $1$, and $\mu_w(Z)=1$.

Let $A$ be in the product $\sigma$-algebra of $S^I$. Then $A$ is $\mu_w$ measurable, and
\[
 \mu_w(Z)=\mu_w(Z\cap A)+\mu_w(Z\setminus A)\leq\mu_w(A)+\mu_w(S^I\setminus A)=1=\mu_w(Z).
\]
Hence $\mu_w(Z\cap A)=\mu_w(A)$. Evaluating $\mu_w=\mu_w\res A+\mu_w\res(S^I\setminus A)$ on sets of the form $T\cap Z$ shows that $A$ is $\mu_w\res Z$ measurable, and $Z$ is $\mu_w\res Z$ measurable trivially. Thus, the displayed set function is the restriction of the measure $(\mu_w\res Z)|_{\cM_{\mu_w\res Z}}$ to a $\sigma$-algebra contained in $\cM_{\mu_w\res Z}$, and it is a measure by Theorem \ref{t:measure-outer-correspondence}(1).
\end{proof}

When $I$ is uncountable and $U\neq S$, the set $U^I$ does not belong to the product $\sigma$-algebra, because every set in that $\sigma$-algebra is determined by countably many coordinates. The proof of Lemma \ref{l:full-outer-measure} uses only the outer measure.

\begin{theorem}[Kolmogorov]\label{t:kolmogorov}
Let $I$ be a nonempty index set of arbitrary cardinality. For every consistent family $(\mu_F)$ of Borel probability measures on the spaces $\RR^F$, there is a unique probability measure $\mu$ on the product $\sigma$-algebra on $\RR^I$ such that
\[
 \mu\bigl(\pi_F^{-1}(B)\bigr)=\mu_F(B)
\]
for every finite $F\subset I$ and every Borel set $B\subset\RR^F$.
\end{theorem}

\begin{proof}
Fix a homeomorphism $\phi:\RR\to(0,1)$. For finite $F\subset I$, let $\phi_F:\RR^F\to(0,1)^F$ act coordinatewise, and define a Borel probability measure $\nu_F$ on $[0,1]^F$ by
\[
 \nu_F(B):=\mu_F\bigl(\phi_F^{-1}(B\cap(0,1)^F)\bigr).
\]
The family $(\nu_F)$ is again consistent, since the maps $\phi_F$ commute with the projections $\pi_F^G$. Let $\mu_w$ be the outer measure on $[0,1]^I$ given by Lemma \ref{l:kolmogorov-compact} with $S=[0,1]$, and put $Z=(0,1)^I$. Since $\nu_{\{i\}}((0,1))=1$ for every $i$, Lemma \ref{l:full-outer-measure} gives a probability measure $A\cap Z\mapsto\mu_w(A)$ on the trace on $Z$ of the product $\sigma$-algebra of $[0,1]^I$. This trace $\sigma$-algebra contains the cylinders of $Z$, and hence the product $\sigma$-algebra of $Z$.

The coordinatewise map $\Phi:\RR^I\to Z$, $\Phi(x)_i=\phi(x_i)$, is a bijection, and $\Phi$ and $\Phi^{-1}$ are measurable for the product $\sigma$-algebras. Transporting the trace measure by $\Phi^{-1}$ gives a probability measure $\mu$ on the product $\sigma$-algebra of $\RR^I$. Since $\Phi(\pi_F^{-1}(B))=\pi_F^{-1}(\phi_F(B))\cap Z$,
\[
 \mu\bigl(\pi_F^{-1}(B)\bigr)=w\bigl(\pi_F^{-1}(\phi_F(B))\bigr)
 =\nu_F(\phi_F(B))=\mu_F(B)
\]
for every finite $F\subset I$ and Borel set $B\subset\RR^F$.

For uniqueness, assign to each cylinder $\pi_F^{-1}(B)\subset\RR^I$ the weight $\mu_F(B)$. This weight is countably subadditive because the probability measure just constructed extends it. The common-coordinate argument with $H=F\cup G$ verifies the hypothesis of Proposition \ref{p:method-one}(M), and $\RR^I$ itself is a cylinder of weight $1$. Lemma \ref{l:uniqueness} therefore gives uniqueness on the product $\sigma$-algebra.
\end{proof}

\section{Mass distributions and Frostman's lemma}\label{s:mass-distribution}

Suppose that you want to define a probability measure $\mu$ on a compact metric space $\XX$, e.g.~normalized Lebesgue measure on $[0,1]^n$. The following inductive procedure is quite natural, but has pitfalls. Put $\Delta_0=\{\XX\}$ and assign $w(\XX)=1$. Suppose that for some $k\geq 0$ you have specified the value of $w$ on a finite partition $\Delta_k$ of $\XX$ into nonempty sets of diameter at most $2^{-k}\diam\XX$. Partition each $Q\in\Delta_k$ into finitely many nonempty sets $\Child(Q)$ of diameter at most $2^{-(k+1)}\diam\XX$, and distribute the mass $w(Q)$ to the children of $Q$ in any way such that
\begin{equation}\label{w-sum}
 w(Q)=\sum_{R\in\Child(Q)}w(R).
\end{equation}
Put $\Delta_{k+1}=\bigcup_{Q\in\Delta_k}\Child(Q)$ and $\mathcal T=\bigcup_{k=0}^\infty\Delta_k$. By Lemma \ref{l:method-one},
\begin{equation}\label{out-A}
 \mu(A)=\inf\left\{\sum_{i=1}^\infty w(Q_i):Q_i\in\mathcal T\text{ and }A\subset\bigcup_{i=1}^\infty Q_i\right\}\quad\text{for all }A\subset\XX
\end{equation}
is an outer measure on $\XX$, and using \eqref{w-sum} one can show that Borel sets are $\mu$ measurable (see Theorem \ref{t:existence}). The point is that the finite conservation law \eqref{w-sum} does not by itself solve the extension problem: it is not at all clear and may not be true that $\mu(Q)=w(Q)$ for $Q\in\mathcal T$, or even that $\mu(\XX)=1$. Example \ref{ex:not-tame} shows that $\mu(\XX)=0$ is possible. Theorem \ref{t:existence} characterizes when the continuity principle upgrades this finite information to the countable subadditivity needed for faithful extension, without taking limits of measures.

\begin{definition}[nested partitions]\label{d:nested-partitions}
Let $\XX$ be a proper metric space. A sequence $(\Delta_k)_{k\geq0}$ of families of subsets of $\XX$ is a \emph{sequence of nested partitions} if the following conditions hold.
\begin{enumerate}
\item Every $\Delta_k$ is a partition of $\XX$ into nonempty bounded Borel sets.
\item Every set in $\Delta_{k+1}$ is contained in a set in $\Delta_k$.
\item Every bounded set in $\XX$ intersects only finitely many sets in $\Delta_k$ for every $k\geq 0$.
\item $\sup_{Q\in\Delta_k}\diam Q\to0$ as $k\to\infty$.
\end{enumerate}
\end{definition}

Conditions (3) and (4) imply that every bounded set in $\XX$ is totally bounded. Hence a complete metric space that admits a sequence of nested partitions is proper.

Let $(\Delta_k)_{k\geq0}$ be a sequence of nested partitions. We call the sets in $\mathcal T:=\bigcup_{k\geq0}\Delta_k$ \emph{cubes}. By (3), every $\Delta_k$ is countable. For $Q\in\Delta_k$ and $n\geq0$, let $\Child^n(Q)$ be the finite set of cubes in $\Delta_{k+n}$ contained in $Q$, and put $\Child(Q):=\Child^1(Q)$. If a set $Q$ belongs to both $\Delta_j$ and $\Delta_k$ with $j<k$, then $Q$ belongs to $\Delta_i$ for all $j\leq i\leq k$, and $\Child(Q)=\{Q\}$ at the levels $j\leq i<k$. The boundary of a cube $Q$ is denoted by $\partial Q$, and for $Q\in\Delta_k$ and $n\geq0$, we put
\[
 \Collar^n(Q):=\{R\in\Child^n(Q):\overline R\cap\partial Q\neq\emptyset\}.
\]
We call $\Collar^n(Q)$ the \emph{collar} of $Q$ at depth $n$. It consists of the descendants of $Q$ at depth $n$ whose closures meet the boundary of $Q$.
When a set $Q$ occurs at more than one level, the symbols $\Child^n(Q)$ and $\Collar^n(Q)$ are understood relative to the specified occurrence $Q\in\Delta_k$. The mass $w(Q)$ below depends only on the underlying set.

We will repeatedly use the following consequence of (3). At each fixed level $k$, the family $\{\overline Q:Q\in\Delta_k\}$ is locally finite. Indeed, if a bounded set $A$ meets $\overline Q$, then $Q$ meets the bounded $1$-neighborhood of $A$, and only finitely many cubes in $\Delta_k$ can do so. Consequently, only finitely many level-$k$ closures meet a fixed bounded set, and the union of any subfamily of level-$k$ closures is closed.

\begin{definition}[mass distributions]\label{d:mass-distribution}
Let $(\Delta_k)_{k\geq0}$ be a sequence of nested partitions of a proper metric space $\XX$. A \emph{mass distribution} is a function $w:\mathcal T\to[0,\infty)$ such that
\begin{equation}\label{mass-additivity}
 w(Q)=\sum_{R\in\Child(Q)}w(R)\quad\text{for all }Q\in\Delta_k\text{ and }k\geq0.
\end{equation}
A mass distribution $w$ is \emph{tame} if
\begin{equation}\label{tame-condition}
 \lim_{n\to\infty}\sum_{R\in\Collar^n(Q)}w(R)=0\quad\text{for every occurrence }Q\in\Delta_k.
\end{equation}
Given a mass distribution $w$, put $w(\emptyset):=0$ and let $\mu_w$ be the Method I outer measure with weight $w$ on $\mathcal E:=\mathcal T\cup\{\emptyset\}$.
\end{definition}

If a cube $Q$ belongs to two consecutive levels, then $\Child(Q)=\{Q\}$ at the lower level, and \eqref{mass-additivity} reads $w(Q)=w(Q)$. Thus, \eqref{mass-additivity} is consistent with regarding $w$ as a function of the underlying set. Iterating \eqref{mass-additivity} gives $w(Q)=\sum_{R\in\Child^n(Q)}w(R)$ for all $Q\in\Delta_k$ and $n\geq 0$.

\begin{theorem}[existence of measures from mass distributions]\label{t:existence}
Let $(\Delta_k)_{k\geq0}$ be a sequence of nested partitions of a proper metric space $\XX$, and let $w$ be a mass distribution. Then $\mu_w$ is a Radon measure on $\XX$. Moreover, $w$ is tame if and only if
\begin{equation}\label{tame-values}
 \mu_w(Q)=w(Q)\quad\text{and}\quad\mu_w(\partial Q)=0\quad\text{for all }Q\in\mathcal T.
\end{equation}
\end{theorem}

\begin{proof}
\emph{$\mu_w$ is a Radon measure.} By (1) and (3) in Definition \ref{d:nested-partitions}, every bounded set is covered by finitely many cubes in $\Delta_0$, and $\mu_w$ is finite on bounded sets. Let $Q,Q'\in\mathcal T$. If $Q'\subset Q$, take $F_0=Q'$ and $F_1=\emptyset$. If $Q'\cap Q=\emptyset$, take $F_0=\emptyset$ and $F_1=Q'$. Otherwise, $Q\subsetneq Q'$, and we take $F_0=Q$ and let $F_1,\dots,F_p$ be the other cubes at a level of $Q$ that are contained in $Q'$. In each case, the iterated form of \eqref{mass-additivity} gives $\sum_{j=0}^pw(F_j)=w(Q')$. The remaining cases with $\emptyset\in\mathcal E$ are immediate. By Proposition \ref{p:method-one}(M), every cube is $\mu_w$ measurable. Since $\Delta_0$ is a countable partition of $\XX$, $\XX\in\mathcal E_\sigma$. If $U\subset\XX$ is open and $x\in U$, then by (4) in Definition \ref{d:nested-partitions} some cube containing $x$ is contained in $U$. Hence $U$ is the union of the countably many cubes contained in $U$, and every Borel set is $\mu_w$ measurable. By Proposition \ref{p:method-one}(OR), every set $S\subset\XX$ is contained in a set $B\in\mathcal E_{\sigma\delta}$ with $\mu_w(B)=\mu_w(S)$, and $B$ is Borel.

\emph{\eqref{tame-values} implies tameness.} Let $Q\in\Delta_k$ and put $A_n:=\bigcup_{R\in\Collar^n(Q)}R$. If $R\in\Collar^{n+1}(Q)$, then the cube in $\Child^n(Q)$ containing $R$ belongs to $\Collar^n(Q)$. Hence $A_{n+1}\subset A_n$. If $x\in\bigcap_{n}A_n$, then for every $n$ the point $x$ lies in a cube $R\in\Delta_{k+n}$ with $\overline R\cap\partial Q\neq\emptyset$. Since $\diam\overline R=\diam R$, condition (4) shows that points of $\partial Q$ lie arbitrarily close to $x$. Hence $x\in\partial Q$, because $\partial Q$ is closed. The cubes are Borel sets, and $\mu_w$ is a Borel measure. By \eqref{tame-values} and continuity from above,
\[
 \lim_{n\to\infty}\sum_{R\in\Collar^n(Q)}w(R)=\lim_{n\to\infty}\mu_w(A_n)=\mu_w\Big(\bigcap_nA_n\Big)\leq\mu_w(\partial Q)=0.
\]

\emph{Tameness implies \eqref{tame-values}.} Assume that $w$ is tame. We verify the hypotheses of Lemma \ref{l:continuity}.

(1) Let $Q,Q_1,\dots,Q_l\in\mathcal T$ with $Q\subset Q_1\cup\dots\cup Q_l$, and choose $n$ larger than the levels of $Q,Q_1,\dots,Q_l$. Every cube in $\Delta_n$ contained in $Q$ meets some $Q_i$, and is therefore contained in $Q_i$. Hence
\[
 w(Q)=\sum_{R\in\Delta_n,\ R\subset Q}w(R)\leq\sum_{i=1}^l\sum_{R\in\Delta_n,\ R\subset Q_i}w(R)=\sum_{i=1}^lw(Q_i).
\]

Fix $Q\in\Delta_k$ and $\varepsilon>0$. By the local finiteness observation above, the set $\mathcal N(Q)$ of cubes $P\in\Delta_k$ with $P\neq Q$ and $\overline P\cap\overline Q\neq\emptyset$ is finite. By \eqref{tame-condition}, choose $n\geq 1$ such that
\[
 \sum_{R\in\Collar^n(Q)}w(R)\leq\varepsilon\quad\text{and}\quad\sum_{P\in\mathcal N(Q)}\sum_{R\in\Collar^n(P)}w(R)\leq\varepsilon.
\]

(2) Let $\mathcal F$ be the set of cubes $S\in\Delta_{k+n}$ with $S\cap Q=\emptyset$ and $\overline S\cap\overline Q\neq\emptyset$. This set is finite by local finiteness, since $\overline Q$ is bounded. Let $S\in\mathcal F$, let $P\in\Delta_k$ be the cube containing $S$, and let $y\in\overline S\cap\overline Q$. Then $P\neq Q$ and $y\in\overline P\cap\overline{\XX\setminus P}=\partial P$. Hence $P\in\mathcal N(Q)$ and $S\in\Collar^n(P)$. It follows that $\sum_{S\in\mathcal F}w(S)\leq\varepsilon$. By the same local finiteness observation, the union of the closures of the cubes $S\in\Delta_{k+n}$ with $\overline S\cap\overline Q=\emptyset$ is closed. Its complement $U$ is an open set containing $\overline Q$. Every point of $U\setminus Q$ lies in a cube $S\in\Delta_{k+n}$ with $S\cap Q=\emptyset$ and $\overline S\cap\overline Q\neq\emptyset$. Hence $U\setminus Q\subset\bigcup\mathcal F$.

(3) The family $\Child^n(Q)$ is finite by (3) in Definition \ref{d:nested-partitions}, since all of its cubes meet the bounded set $Q$. Let $K$ be the union of the closures of the cubes in $\Child^n(Q)\setminus\Collar^n(Q)$. Each closure is compact because $\XX$ is proper, so $K$ is compact. Every such closure lies in $\overline Q\setminus\partial Q$. Hence $K\subset Q$ and $Q\setminus K\subset\bigcup\Collar^n(Q)$, which has total weight at most $\varepsilon$.

By Lemma \ref{l:continuity}, $\mu_w(Q)=w(Q)$ for all $Q\in\mathcal T$. Finally, $\partial Q\cap Q\subset Q\setminus K$ and $\partial Q\setminus Q\subset U\setminus Q$. Thus, $\mu_w(\partial Q)\leq 2\varepsilon$ for every $\varepsilon>0$.
\end{proof}

\begin{example}\label{ex:not-tame}
Tameness cannot be dropped. Let $\XX=[0,1]$ and $\Delta_0=\{\XX\}$, and for $k\geq1$ let $\Delta_k$ consist of $[0,2^{-k}]$ and the intervals $(i2^{-k},(i+1)2^{-k}]$ with $1\leq i<2^k$. Define $w(\XX)=1$, $w((1/2,1/2+2^{-k}])=1$ for every $k\geq1$, and $w(Q)=0$ for all other cubes. Then $w$ is a mass distribution. Every point of $[0,1/2]$ lies in the cube $[0,1/2]$ of weight $0$. If $x>1/2$ and $2^{1-k}<x-1/2$, then the cube in $\Delta_k$ containing $x$ has left endpoint greater than $1/2$ and weight $0$. Hence $\mu_w(\XX)=0\neq w(\XX)$. Here $\partial(1/2,1]=\{1/2\}$, the cube $(1/2,1/2+2^{-n-1}]$ belongs to $\Collar^n((1/2,1])$ for all $n$, and $w$ is not tame.
\end{example}

For $A,B\subset\XX$, the \emph{gap} between $A$ and $B$ is
\[
 \gap(A,B):=\inf\{d(a,b):a\in A,\ b\in B\},
\]
where $\inf\emptyset=\infty$.

\begin{corollary}[a tameness criterion]\label{c:tame}
Let $w$ be a mass distribution for a sequence of nested partitions of a proper metric space $\XX$. Suppose that there is $\eta>0$ such that, for every occurrence $Q\in\Delta_k$, there is a family $\mathcal I(Q)\subset\Child(Q)$ with
\[
 \gap\Big(\bigcup\mathcal I(Q),\XX\setminus Q\Big)>0\quad\text{and}\quad\sum_{R\in\mathcal I(Q)}w(R)\geq\eta\,w(Q).
\]
Then
\begin{equation}\label{tame-geometric}
 \sum_{R\in\Collar^n(Q)}w(R)\leq(1-\eta)^n w(Q)\quad\text{for all }Q\in\mathcal T\text{ and }n\geq 0.
\end{equation}
In particular, $w$ is tame, and \eqref{tame-values} holds.
\end{corollary}

A preliminary version of Corollary \ref{c:tame} appears in Badger and Schul \cite[Lemma 4.14]{BS4}, for systems of metric $b$-adic cubes with $b>5$, where $\mathcal I(Q)$ consists of the central child of $Q$. There, the countable subadditivity of the weight is verified directly by a compactness argument, which Lemma \ref{l:continuity} isolates.

\begin{proof}
If $w\equiv0$, the conclusion is immediate. Otherwise the hypothesis forces $\eta\leq1$, and we assume $0<\eta\leq1$. Fix $Q\in\mathcal T$. Put $\mathcal A_0:=\{Q\}$ and $\mathcal A_{n+1}:=\bigcup_{P\in\mathcal A_n}(\Child(P)\setminus\mathcal I(P))$ for all $n\geq 0$. Then $\sum_{R\in\mathcal A_n}w(R)\leq(1-\eta)^nw(Q)$ by induction. Let $R\in\Child^n(Q)\setminus\mathcal A_n$. Then there are $0\leq i<n$ and cubes $P\in\Child^i(Q)$ and $P'\in\mathcal I(P)$ with $R\subset P'$. Since $\XX\setminus Q\subset\XX\setminus P$, we have $\gap(R,\XX\setminus Q)\geq\gap(P',\XX\setminus P)>0$. Hence $\overline R\cap\partial Q=\emptyset$, and $R\notin\Collar^n(Q)$. Thus, $\Collar^n(Q)\subset\mathcal A_n$, which gives \eqref{tame-geometric}. Theorem \ref{t:existence} completes the proof.
\end{proof}

If a cube $Q$ belongs to two consecutive levels, then $\Child(Q)=\{Q\}$ at the lower level. At that occurrence, the hypothesis of Corollary \ref{c:tame} holds only if $w(Q)=0$ or $\gap(Q,\XX\setminus Q)>0$. Thus, the criterion is suited to nested partitions in which cubes of positive mass do not repeat across levels, such as the $b$-adic cubes below.

\begin{example}[$b$-adic cubes]\label{ex:b-adic}
Fix an integer $b\geq3$ and let
\[
 \Delta_k=\left\{b^{-k}\bigl(m+(0,1]^n\bigr):m\in\ZZ^n\right\},
 \qquad k\geq0.
\]
These are nested partitions of $\RR^n$.  Every cube $Q$ has a child
$Q^\circ$ whose closure is contained in the interior of $Q$. When $b$ is
odd, one may take the central child.  If a mass distribution satisfies
$w(Q^\circ)\geq\eta w(Q)$ for some $\eta>0$ independent of $Q$, then
Corollary \ref{c:tame} applies with $\mathcal I(Q)=\{Q^\circ\}$.  Hence the
Method I measure has the prescribed masses and gives zero mass to every cube
boundary.  The same conclusion holds if the fixed fraction $\eta$ is carried
by any collection of children separated from $\RR^n\setminus Q$.
\end{example}

\begin{corollary}[uniqueness]\label{c:uniqueness}
Let $(\Delta_k)_{k\geq0}$ be a sequence of nested partitions of a proper metric space $\XX$, and let $w$ be a tame mass distribution. If $\nu$ is a Radon measure on $\XX$ with $\nu(Q)=w(Q)$ for all $Q\in\mathcal T$, then $\nu=\mu_w$.
\end{corollary}

\begin{proof}
The family $\mathcal E=\mathcal T\cup\{\emptyset\}$ satisfies the hypothesis of Proposition \ref{p:method-one}(M), as in the proof of Theorem \ref{t:existence}, and $w$ is countably subadditive because $\mu_w$ extends $w$ (Remark \ref{r:extension}). The countably many cubes in $\Delta_0$ cover $\XX$ and have finite weight. By Lemma \ref{l:uniqueness}, $\nu=\mu_w$ on the $\sigma$-algebra generated by $\mathcal T$, which contains every open set by (4) in Definition \ref{d:nested-partitions}, and hence every Borel set. Since $\nu$ and $\mu_w$ are Borel regular, they agree on every subset of $\XX$.
\end{proof}

\begin{example}[completeness is needed for prescribed values]\label{ex:rationals}
Let $\XX=\QQ$ with the usual metric, and let $\Delta_k$ consist of the sets $\QQ\cap(i3^{-k},(i+1)3^{-k}]$ with $i\in\ZZ$. These partitions satisfy (1)--(4) in Definition \ref{d:nested-partitions}, but $\QQ$ is not proper. Put $w(Q)=3^{-k}$ for $Q\in\Delta_k$. Only the two extreme descendants of a cube can meet its boundary, and hence $w$ is tame. On the other hand, every rational number lies in cubes of arbitrarily small weight. Hence $\mu_w(\QQ)=0$, while $w(\QQ\cap(0,1])=1$, as in Example \ref{ex:rational-intervals}. In the proof of Theorem \ref{t:existence}, the argument breaks down at condition (3) of Lemma \ref{l:continuity}, because closures of cubes in $\QQ$ are not compact.
\end{example}

\subsection{Frostman's lemma}\label{ss:frostman}

The mass-distribution theorem gives a direct proof of Frostman's lemma when
the exponent is larger than the dimension of the boundaries in a cubical
partition.  We formulate the construction using net content, with side length
rather than diameter as the cost of a cube.  The construction begins with
infinite-depth covering contents and distributes mass forward through the
cubical tree.  Thus, it does not pass through measures on finite trees or use
weak compactness of measures.

For $s>0$ and $E\subset\RR^n$, recall that
\[
 \mathcal H^s_\infty(E)
 :=\inf\left\{\sum_{i=1}^\infty(\diam U_i)^s:
 E\subset\bigcup_{i=1}^\infty U_i\right\}.
\]

A \emph{cube} in $\RR^n$ is a product $Q=J_1\times\dots\times J_n$ of bounded intervals $J_1,\dots,J_n$ of a common length $\ell>0$, each of which may be open, closed, or half-open. We write $\side Q:=\ell$. Thus, $\diam Q=\sqrt n\,\side Q$.

\begin{definition}[nets and net content]\label{d:nets}
Let $n\geq1$ and $m\geq2$ be integers, and let $t>0$.  An \emph{$m$-adic
net} $\Delta$ \emph{at initial scale $t$} is a family of cubes in $\RR^n$
such that
\begin{enumerate}
\item every cube $Q\in\Delta$ has side length $tm^{-k}$ for some integer
$k\geq0$;
\item the family $\Delta_k$ of cubes in $\Delta$ with side length $tm^{-k}$
is pairwise disjoint;
\item for every $Q\in\Delta_k$ there is a family
$\Child(Q)\subset\Delta_{k+1}$ such that
$Q=\bigcup\Child(Q)$, and
$\Delta_{k+1}=\bigcup_{Q\in\Delta_k}\Child(Q)$.
\end{enumerate}
We call $\Delta_k$ the \emph{$k$-th level} of $\Delta$.  We do not require
the cubes in a net to have a fixed type. In particular, the choices of
boundary faces may vary from cube to cube.  A $2$-adic net is \emph{dyadic}
and a $3$-adic net is \emph{triadic}.  For $Q\in\Delta$, let $\Delta(Q)$
denote the family consisting of $Q$ and all of its descendants.

For $s\geq0$ and $E\subset\RR^n$, the \emph{$s$-dimensional $\Delta$
content} is
\begin{equation}\label{net-content}
 \mathcal N^{\Delta,s}_\infty(E)
 :=\inf\left\{\sum_{i=1}^\infty(\side Q_i)^s:
 E\subset\bigcup_{i=1}^\infty Q_i,\quad Q_i\in\Delta\right\},
\end{equation}
where the infimum is $\infty$ if no such cover exists.
\end{definition}

Thus, $\mathcal N^{\Delta,s}_\infty$ is the Method I outer measure obtained
from the net with weight $Q\mapsto(\side Q)^s$.

The \emph{center} of a cube $Q=J_1\times\dots\times J_n$ is the point whose
$i$th coordinate is the midpoint of $J_i$. Two elementary counting facts about
a level $\Delta_k$ of an $m$-adic net, with $\ell=tm^{-k}$, will be used
below. First, the centers $c$ and $c'$ of distinct cubes in $\Delta_k$ satisfy
$\max_i|c_i-c_i'|\geq\ell$, since otherwise the point $(c+c')/2$ lies in both
cubes. An open cube of side less than $3\ell$ is the union of $3^n$ cubes of
side less than $\ell$, and each of these contains at most one center. Hence
\begin{equation}\label{center-count}
 \text{an open cube of side less than }3\ell\text{ contains the centers of at
 most }3^n\text{ cubes in }\Delta_k.
\end{equation}
Second, the children of a cube $Q\in\Delta_k$ are disjoint cubes of side
$\ell/m$ with union $Q$. Comparing Lebesgue measure shows that $Q$ has exactly
$m^n$ children, and hence $m^{jn}$ descendants in $\Delta_{k+j}$.

\begin{theorem}[Frostman's lemma for $s>n-1$]\label{t:frostman}
Let $n\geq1$ and $m\geq2$ be integers, and let $n-1<s\leq n$. Let $\Delta$ be
an $m$-adic net at initial scale $t>0$, and let $E\subset\RR^n$ be a compact
set contained in a cube $Q_0\in\Delta_0$. Then there is a finite Radon measure
$\mu$ supported on $E$ such that
\begin{align}
 \mu(Q)&\leq(\side Q)^s
 &&\text{for every }Q\in\Delta, \label{frostman-net-cubes}\\
 \mu(E)&=\mathcal N^{\Delta,s}_\infty(E)
 \geq n^{-s/2}\mathcal H^s_\infty(E), \label{frostman-net-mass}\\
 \mu(B(x,r))&\leq 3^nm^sr^s
 &&\text{for every }x\in\RR^n\text{ and }r>0. \label{frostman-ball-bound}
\end{align}
Consequently, $\nu:=3^{-n}m^{-s}\mu$ satisfies the usual Frostman bound
$\nu(B(x,r))\leq r^s$ together with
\[
 \nu(E)\geq \frac{1}{3^nm^sn^{s/2}}\,\mathcal H^s_\infty(E).
\]
\end{theorem}

Every compact set $E\subset\RR^n$ lies in a level-zero cube of some
$m$-adic net, for example a translated grid
$\{a+tm^{-k}(p+(0,1]^n):p\in\ZZ^n,\ k\geq0\}$ with $a$ and $t$ chosen
suitably. Thus,
Theorem \ref{t:frostman} contains the usual form of Frostman's lemma for
$s>n-1$. The hypothesis that $E$ lies in a single cube of $\Delta_0$ is used
only for the ball bound with $r\geq t$, where it gives $\mu(E)\leq t^s$. Net
content is additive across the cubes of $\Delta_0$, and without this
hypothesis, $\mathcal N^{\Delta,s}_\infty(E)$ can be much larger than
$(\diam E)^s$.

\begin{proof}
Every cube in $\Delta$ has an ancestor in $\Delta_0$, by (3) in Definition
\ref{d:nets}, and the cubes in $\Delta_0$ are pairwise disjoint. Hence a cube
in $\Delta$ that is not a descendant of $Q_0$ is disjoint from $Q_0$.

For every descendant $Q$ of $Q_0$, define the localized net content
\begin{equation}\label{localized-content}
 c(Q):=\mathcal N^{\Delta(Q),s}_\infty(E\cap Q).
\end{equation}
Since $Q$ itself is an admissible cover,
\begin{equation}\label{content-bound}
 0\leq c(Q)\leq(\side Q)^s.
\end{equation}

The localized contents satisfy the exact recursion
\begin{equation}\label{content-recursion}
 c(Q)=\min\left\{(\side Q)^s,
          \sum_{R\in\Child(Q)}c(R)\right\}.
\end{equation}
Indeed, covering by $Q$ gives the first upper bound. Taking nearly optimal
covers of $E\cap R$ for the finitely many children gives the second. For the
reverse inequality, consider any cover admitted in \eqref{localized-content}.
If it contains $Q$, its total is at least $(\side Q)^s$. Otherwise every
member of the cover is contained in a unique child $R$ of $Q$, and the
members belonging to $R$ cover $E\cap R$. The total is therefore at least
$\sum_{R\in\Child(Q)}c(R)$.

We now distribute the mass $c(Q_0)$ down the tree. Put $w(Q_0):=c(Q_0)$.
If $w(Q)\leq c(Q)$ has been defined and
$\sum_{R\in\Child(Q)}c(R)>0$, put
\begin{equation}\label{frostman-weights}
 w(R):=w(Q)\,\frac{c(R)}{\sum_{R'\in\Child(Q)}c(R')}
 \quad\text{for }R\in\Child(Q).
\end{equation}
If the denominator vanishes, put $w(R)=0$ for every child $R$. In that
case \eqref{content-recursion} gives $c(Q)=w(Q)=0$. Otherwise the child
weights sum to $w(Q)$, and \eqref{content-recursion} gives
$w(Q)/\sum_{R'}c(R')\leq1$. Thus, $w(R)\leq c(R)$ in either case. Iterating
defines $w$ on every descendant of $Q_0$, and
\begin{equation}\label{frostman-cube-bound}
 w(Q)\leq c(Q)\leq(\side Q)^s\quad\text{for every descendant }Q\text{ of }Q_0.
\end{equation}

To apply Theorem \ref{t:existence}, complete the descendants of $Q_0$ to
nested partitions of $\RR^n$. For $k\geq0$, let $\Delta_k'$ consist of the
level-$k$ descendants of $Q_0$ together with the nonempty sets
$G\setminus Q_0$, where $G=tm^{-k}(p+(0,1]^n)$ and $p\in\ZZ^n$. The level-$k$
descendants of $Q_0$ partition $Q_0$, and the grids at consecutive levels
refine one another. Every bounded set meets only finitely many sets in
$\Delta_k'$, and every set in $\Delta_k'$ has diameter at most $\sqrt n\,tm^{-k}$.
Thus, $(\Delta_k')_{k\geq0}$ is a sequence of nested partitions of $\RR^n$.
Give every set $G\setminus Q_0$ weight zero. Then $w$ is a mass distribution
for $(\Delta_k')_{k\geq0}$, and its only cubes of positive weight are
descendants of $Q_0$.

This mass distribution is tame. The sets $G\setminus Q_0$ have only
descendants of weight zero. Let $Q$ be a descendant of $Q_0$ and put
$\ell_j=m^{-j}\side Q$. If $R\in\Collar^j(Q)$, then $\overline R\subset\overline Q$
and $\overline R$ meets $\partial Q$. Some side interval of $\overline R$
therefore contains an endpoint of the corresponding side interval of
$\overline Q$, and $R$ lies in one of the $2n$ slabs of $\overline Q$ of width
$\ell_j$ adjacent to the faces of $Q$. These slabs have total Lebesgue measure
at most $2n\ell_j(\side Q)^{n-1}$, and the cubes in $\Collar^j(Q)$ are disjoint
with Lebesgue measure $\ell_j^n$. Hence $\Collar^j(Q)$ has at most
$2nm^{j(n-1)}$ members, and
\[
 \sum_{R\in\Collar^j(Q)}w(R)
 \leq 2nm^{j(n-1)}m^{-js}(\side Q)^s
 =2nm^{-j(s-n+1)}(\side Q)^s\longrightarrow0.
\]
Theorem \ref{t:existence} therefore gives a Radon measure $\mu$ on $\RR^n$
such that
\begin{equation}\label{frostman-prescribed-masses}
 \mu(Q)=w(Q)\quad\text{for every }Q\in\bigcup_{k\geq0}\Delta_k'.
\end{equation}

The measure $\mu$ is supported on $E$. The sets $G\setminus Q_0$ cover
$\RR^n\setminus Q_0$ and have measure zero. If $x\in Q_0\setminus E$, then a
sufficiently small descendant of $Q_0$ containing $x$ is disjoint from $E$,
because $E$ is compact. Such a descendant $R$ has $w(R)\leq c(R)=0$. Hence
$\RR^n\setminus E$ is a countable union of $\mu$ null sets. In particular,
every cube in $\Delta$ that is not a descendant of $Q_0$ has measure zero.
Together with \eqref{frostman-cube-bound} and
\eqref{frostman-prescribed-masses}, this proves \eqref{frostman-net-cubes}.
Since the cubes in $\Delta$ that meet $E$ are descendants of $Q_0$,
\[
 \mu(E)=\mu(Q_0)=c(Q_0)=\mathcal N^{\Delta,s}_\infty(E).
\]
If $E\subset\bigcup_iQ_i$ with $Q_i\in\Delta$, then
\[
 \mathcal H^s_\infty(E)
 \leq\sum_i(\diam Q_i)^s
 =n^{s/2}\sum_i(\side Q_i)^s.
\]
Taking the infimum over net covers proves \eqref{frostman-net-mass}.

It remains to pass from cubes to balls. If $0<r<t$, choose $k\geq0$ so that
$r<\ell:=tm^{-k}\leq mr$. The center of every cube in $\Delta_k$ that meets
$B(x,r)$ lies in the open cube with center $x$ and side $2r+\ell<3\ell$. By
\eqref{center-count}, at most $3^n$ cubes in $\Delta_k$ meet $B(x,r)$. Since
$\mu(\RR^n\setminus Q_0)=0$ and the level-$k$ descendants of $Q_0$ partition
$Q_0$, \eqref{frostman-net-cubes} gives
\[
 \mu(B(x,r))\leq3^n\ell^s\leq3^nm^sr^s.
\]
If $r\geq t$, then
$\mu(B(x,r))\leq\mu(E)\leq(\side Q_0)^s=t^s\leq r^s$.
This proves \eqref{frostman-ball-bound} and completes the proof.
\end{proof}

\begin{remark}[all exponents]\label{r:frostman-codimension}
The restriction $s>n-1$ is used only to prove tameness in $\RR^n$. For every
$0<s\leq n$, keep the same net, localized contents, and weights, and work
first on the compact ultrametric space $\Sigma$ of descending chains
$Q_0\supset Q_1\supset\cdots$ with $Q_k\in\Delta_k$, with distance $m^{-k}$
when $k$ is the first index at which two chains differ. The cylinders $[Q]$ of
chains through $Q$ are compact and open, and hence have empty boundary. The
weights $w(Q)$ therefore form a tame mass distribution on $\Sigma$, and
Theorem \ref{t:existence} gives a finite measure $\widetilde\mu$ with
$\widetilde\mu([Q])=w(Q)$. Let $\pi:\Sigma\to\overline{Q_0}$ send a chain to the
unique point of $\bigcap_k\overline{Q_k}$, and let $\mu_0=\pi_\#\widetilde\mu$
be the image of $\widetilde\mu$ under the continuous map $\pi$, that is,
$\mu_0(A)=\widetilde\mu(\pi^{-1}(A))$ for all $A\subset\RR^n$. As in the proof
of Theorem \ref{t:frostman}, $\mu_0$ is a finite Borel measure concentrated on
$E$, and $\mu_0(E)=w(Q_0)=\mathcal N^{\Delta,s}_\infty(E)$.

The only new point is boundary multiplicity. If a chain has $\pi$-image in a
set $A$, then its cube at level $k$ has closure meeting $A$. When $A\in\Delta_k$,
or when $A=B(x,r)$ with $r<\ell=tm^{-k}$, the centers of these cubes lie in an
open cube of side less than $3\ell$, and \eqref{center-count} bounds their
number by $3^n$. Hence $\mu:=3^{-n}\mu_0$ satisfies
\begin{equation}\label{frostman-all-exponents}
 \mu(Q)\leq(\side Q)^s\quad\text{for all }Q\in\Delta,
 \qquad
 \mu(E)=3^{-n}\mathcal N^{\Delta,s}_\infty(E)
 \geq3^{-n}n^{-s/2}\mathcal H^s_\infty(E),
\end{equation}
and choosing $k$ with $r<tm^{-k}\leq mr$ gives $\mu(B(x,r))\leq m^sr^s$ for
$0<r<t$. The case $r\geq t$ follows from the total-mass bound. Thus, the usual
Frostman conclusion holds for every $0<s\leq n$, at the cost of the factor
$3^{-n}$.

For a related use of an ultrametric passage in metric-cube constructions, see
K\"aenm\"aki, Rajala, and Suomala \cite{KRS-cubes}. For tree proofs phrased
in terms of flows and cutsets, see Lyons \cite{Lyons1990} and Bishop and Peres
\cite[Chapter 3]{BishopPeres}.
\end{remark}

\subsection*{AI use disclosure}

The author used GPT-5.6 Sol (OpenAI) and Claude Opus 5 and Claude Opus 5.5 (Anthropic) during the preparation of this manuscript. These tools were used to assist with drafting and revising prose and \LaTeX{} source, comparing alternative formulations, checking mathematical arguments, and organizing references for subsequent verification. The author reviewed the resulting text and proofs, verified the mathematical content, and takes responsibility for the manuscript.

\providecommand{\bysame}{\leavevmode\hbox to3em{\hrulefill}\thinspace}
\providecommand{\MR}{\relax\ifhmode\unskip\space\fi MR }
\providecommand{\MRhref}[2]{%
  \href{http://www.ams.org/mathscinet-getitem?mr=#1}{#2}
}
\providecommand{\href}[2]{#2}

\end{document}